\documentclass[a4paper, reqno, 11pt]{amsart}
\usepackage[usenames,dvipsnames]{color}
\usepackage{xcolor}
\usepackage{setspace}
\usepackage[utf8]{inputenc}
\usepackage[T1]{fontenc}
\usepackage{amsmath,environ}
\usepackage[all]{xy}
\usepackage{amscd}
\usepackage{stmaryrd}
\usepackage{amsfonts}
\usepackage{amssymb}
\usepackage{graphicx}
\usepackage{amsthm}
\usepackage{amsxtra,amsrefs}
\usepackage{enumerate}
\usepackage{leftidx}
\usepackage{hyperref}
\usepackage{extarrows}
\usepackage{wasysym}
\usepackage{mathtools}
\usepackage{url}
\usepackage[OT2,T1]{fontenc}
\usepackage{epstopdf}
\usepackage {hyperref}

\usepackage{mathrsfs} %% to use \mathscr we need this package 

\newcommand\DistTo{\xrightarrow{
   \,\smash{\raisebox{-0.65ex}{\ensuremath{\scriptstyle\sim}}}\,}}
   
\usepackage{forest} %%% use for the diagram of the tree
\forestset{qtree edges/.style={for tree=
    {parent anchor=south, child anchor=north}}}

\graphicspath{ {Image/} }

\newtheorem{theo}{Theorem}
\newtheorem{lem}[theo]{Lemma}
\newtheorem{prop}[theo]{Proposition}
\newtheorem{defi}[theo]{Definition}
\newtheorem{defi-prop}[theo]{Definition-Proposition}

\newtheorem{re}[theo]{Remark}
\newtheorem{ex}[theo]{Example}
\newtheorem{formu}[theo]{Formula}
\newtheorem{note}[theo]{Notation}
\numberwithin{theo}{section}
\newenvironment{altenumerate}
   {\begin{list}
      {\textup{(\theenumi)} }
      {\usecounter{enumi}
       \setlength{\labelwidth}{0pt}
       \setlength{\labelsep}{0pt}
       \setlength{\leftmargin}{0pt}
       \setlength{\itemsep}{\the\smallskipamount}
       \renewcommand{\theenumi}{\roman{enumi}}
      }}
   {\end{list}}
\newenvironment{altitemize}
   {\begin{list}
      {$\bullet$}
      {\setlength{\labelwidth}{0pt}
	   \setlength{\itemindent}{5pt}
       \setlength{\labelsep}{5pt}
       \setlength{\leftmargin}{0pt}
       \setlength{\itemsep}{\the\smallskipamount}
      }}
   {\end{list}}

\usepackage[textwidth=6.5in,textheight=9.0in,centering]{geometry}
\usepackage[small,nohug,heads=vee]{diagrams}
\diagramstyle[labelstyle=\scriptstyle]

\usepackage[nopar]{lipsum}%%% For modifying the margins of math environment

\makeatletter
\NewEnviron{widerequation}{%
  \begin{equation*}
  \sbox\z@{\let\label\@gobble$\displaystyle\BODY$}
  \makebox[\textwidth]{%
    \begin{minipage}{\dimexpr\wd\z@+3em}
    \vspace{-\baselineskip}
    \begin{equation}
    \BODY
    \end{equation}
    \end{minipage}%
  }
  \end{equation*}
}

\begin{document}

\author{XIAOHUA AI}

%\address{Max Planck Institute for Mathematics, Vivatsgasse 7,
%53111 Bonn, Germany}
\address{Max-Planck-Institut für Mathematik, Vivatsgasse 7,
53111 Bonn, Germany}
\email{xiaohua.ai@mpim-bonn.mpg.de}

\date{\today}

\title{GENERALIZED MULTIPLE ZETA VALUES OVER NUMBER FIELDS  I}

\begin{abstract}
Inspired by the theory of Hodge correlators due to Goncharov and by the plectic principle of 
Nekov\'a\v{r} and Scholl, we construct higher plectic Green functions and give a higher order generalization of Hecke's formula for abelian $L$-functions over arbitrary number fields. We hence provide a potential method to generalize multiple zeta values over number fields. We recover classical multiple zeta values and multiple polylogrithms evaluated at roots of unity,  when the number field in consideration is the rational field $\mathbb{Q}$.
\end{abstract}

\maketitle
\tableofcontents

\section{Introduction}
   The objective of this article is to give a potential definition of multiple \lq\lq Dedekind\rq\rq\ zeta values over an arbitrary number field. Classically, the Riemann zeta function is defined as 
\[\zeta(s) = \sum_{\substack{n >0 \\ n\in\mathbb{Z}}}\frac{1}{n^{s}}, \quad \mathrm{Re}(s) > 1.\]
The Riemann zeta function is a special case of the Dirichlet $L$-function associated to the trivial character. There are two directions to generalize the Riemann zeta function. Instead of considering one variable, we have the multiple zeta function defined as
\[\zeta(s_1, \ldots, s_k) = \sum_{\substack{0 < n_1< n_2 < \cdots < n_k\\ n_j \in \mathbb{N} }}\frac{1}{n_1^{s_1}\cdots n_k^{s_k}}, \quad s_j \in \mathbb{C}\]
which converges when $\mathrm{Re}(s_j) + \ldots + \mathrm{Re}(s_k) >k- j +1 $ for all $j$. 
When $s_1, \ldots, s_k$ are all positive integers (with $s_k > 1$) these sums are often called multiple zeta values (MZVs). The $k$ in the definition is called the depth (or length) of a MZV, and the sum 
$m = \sum^{k}_{j=1} s_j$ is called the weight.

   The second direction of generalization is to replace the rational field $\mathbb{Q}$ by an arbitrary number field $K$. The Dedekind zeta function of $K$ is defined as 
\[\zeta_{K} (s) = \sum_{a\subset O_{K}} \frac{1}{(N_{K/\mathbb{Q}}(a))^{s}},\]
where $a$ ranges through the non-zero ideals of the ring of integers $O_{K}$ of $K$ and $N_{K/\mathbb{Q}}(a)$ denotes the absolute norm of $a$. This sum is absolutely convergent for $\mathrm{Re}(s)>1$.

  Moreover when we fix an ideal $I$ of $K$, we can define the partial Dedekind zeta function
  \[\zeta_{K,I}(s) = \sum_{a\in I/O^{\times}_{K}}\frac{1}{\vert N_{K/\mathbb{Q}}(a)\vert ^{s}},\]
  where $a$ still runs through all non-zero elements in $I/O^{\times}_{K}$.

   However little is known if we combine the two directions of generalization. It is natural to ask what we should put in the missing place in the following diagram to complete this diagram.
\begin{diagram}
\mathrm{Classical}\  \mathrm{zeta}\ \mathrm{values}/\mathbb{Q} &\rTo &  (\mathrm{partial}) \mathrm{Dedekind}\ \mathrm{zeta}\ \mathrm{functions}/K \\
 \dTo &     & \dDotsto \\
\mathrm{Mutiple}\ \mathrm{zeta}\ \mathrm{values}/\mathbb{Q} & \rDotsto &  \textbf{?} 
\end{diagram}

 In this paper, we will provide a potential answer to this question. Our starting point is to generalize the Hecke formula. Hecke's formula is one of the typical examples within the theory of automorphic periods, which relates the $L$-functions and period integrals. 
 
\subsection{Hecke's formula} 
 In 1917, Hecke \cite{Hec17} proved that the integral of the restriction of a suitable Eisenstein series on $GL(n)$ over $\mathbb{Q}$ to the idele class group of a given number field (of degree $n$) multiplied by an idele class character $\chi$ of finite order is equal to the $L$-functions of $\chi$, up to some $\Gamma$- factors. 

 More precisely, Let $K$ be a number field of degree $[K: \mathbb{Q}] = r_1+ 2 r_2$,
\[K_{\mathbb{R}} = K\otimes \mathbb{R} \DistTo \mathbb{R}^{r_1}\times\mathbb{C}^{r_2}. \]
Define the norm map 
\[N = N_{K/\mathbb{Q}}\otimes id : K^{\times}_{\mathbb{R}} \longrightarrow \mathbb{R}^{\times}.\]
In order to state Hecke's formula, we will need the following data:
\begin{altenumerate}
\item  Let $U \subset O^{\times}_{K,+}$ be a subgroup of finite index, where 
$$O^{\times}_{K,+} = O^{\times}_{K} \cap \left( K^{\times}_{\mathbb{R}}\right)_{+}, \quad  \left( K^{\times}_{\mathbb{R}}\right)_{+} = \left(\mathbb{R}^{\times}_{+} \right)^{r_1}\times \left( \mathbb{C}^{\times}\right)^{r_2}.$$ 
\item  Let $I \subset K$ be a fractional $O_{K}$-ideal. 
\item  $\exists\ m\in \mathbb{N}\smallsetminus \{0\}$, $ \quad \phi: I/m I  \longrightarrow \mathbb{C}$ be a function such that 
\[\forall \epsilon \in U, \quad \forall \alpha \in I\smallsetminus\{0\}, \qquad \phi(\epsilon \alpha) = \phi(\alpha) .\]
\end{altenumerate}
We consider the following embedding (defined up to a conjugation) 
\[GL_{K}(1) \hookrightarrow GL_{\mathbb{Q}}([K:\mathbb{Q}]).\]
Let $E(g,s,\phi)$ be the Eisenstein series defined by 
\[E(g, s,\phi) = \sum_{x \in I \smallsetminus \{0\}}\frac{\phi(x)}{\Vert g\cdot x \Vert^{s}}, \]
where $g \in GL_{\mathbb{Z}}(I)(\mathbb{R}) \cong GL_{\mathbb{Q}}([K:\mathbb{Q}])(\mathbb{R})$ and $\Vert \cdot \Vert $ is the standard hermitian norm on $\mathbb{R}^{r_1} \times \mathbb{C}^{r_2} = K_R$.
 
 Hecke proved the following formula
  \begin{theo}[\textbf{The Hecke Formula} \cite{Hec17}]  
 If $$U_{\mathbb{R}} = \mathrm{Ker}\left(N_{K/\mathbb{Q}}\otimes 1: \left( K^{\times}_{\mathbb{R}}\right)_{+} \longrightarrow \mathbb{R}^{\times}_{+} \right),$$ then
\[\int_{U_{\mathbb{R}}/U}E(u, [K:\mathbb{Q}]s, \phi)\ \mathrm{d}^{\times}\mu(u) = \frac{2^{1-r_1}\pi^{r_2}}{[K:\mathbb{Q}]} \frac{\Gamma(s/2)^{r_1}\Gamma(s)^{r_2}}{\Gamma([K:\mathbb{Q}]s/2)}\sum_{\alpha \in (I\setminus \{0\})/U}\frac{\phi(\alpha)}{\mid N_{K/\mathbb{Q}}(\alpha) \mid ^{s}},\]
 where $\mathrm{d}^{\times}\mu(u)$ is a suitably normalized Haar measure on $U_{\mathbb{R}}/U$.
\end{theo}

  This formula is a starting point of our attempt to define the \lq\lq multiple Dedekind zeta values\rq\rq \  in terms of certain period integral in the spirit of Hecke. The problem now is how to write a period integral for multiple variables. The idea is to find an appropriate object to replace the Eisenstein series in Hecke's formula. For this purpose, we get inspiration from the theory of A.~Goncharov \cite{Gon16} on Hodge correlators. 
  
\subsection{Hodge correlators}  
  In his survey of European Congress of Mathematics\cite{Gon00}, Goncharov discussed the problems of the study of the Lie algebra of the image of the motivic Galois group acting on the motivic fundamental group of $\mathbb{P}^{1}\backslash \{0, \mu_{N}, \infty\}$, and mentioned a surprising and mysterious connection between these problems and the geometry of modular varieties. The Hodge realization of these problems is related to the arithmetic problem of multiple zeta values. 
  
  %$l$-adic realization is related to the action of the absolute Galois group $Gal(\overline{\mathbb{Q}}/\mathbb{Q})$ on the pro-$l$-completion $\pi^{(l)}_{1}(\mathbb{P}^{1}\backslash \{0, \mu_{N}, \infty\})$. 

  At the end of this survey Goncharov stated that he expected a similar theory for an arbitrary number field. He also gave an example when the number field is an imaginary quadratic field and he considered the motivic fundamental group of a CM elliptic curve, and constructed the multiple Hecke $L$-values as values at torsion points of multiple elliptic polylogarithms. His intriguing insight is that one can define the multiple polylogarithms for arbitrary curves as correlators for certain Feynman integrals. This idea has been realized in his recent work \cite{Gon16} about the Hodge correlators. 
  
  The Hodge correlators are constructed from just one fundamental object, namely, the Green function, integrated along some Feynman diagrams in Hodge-theoretic setting. In Goncharov's work \cite{Gon16}, Hodge correlators are periods of motivic correlators. $L$-values can be interpreted as Hodge correlators (e.g. the Rankin-Selberg integrals). It seems quite reasonable to take Goncharov's path for our purpose. 
  
  However, Goncharov's construction is carried out over complex field $\mathbb{C}$. In order to apply Goncharov's idea, we need to find the appropriate analogue for number fields. That's the moment when Nekov\'a\v r and Scholl's plectic principle comes in. 
  
\subsection{Plectic principle} 

  Let $F$ be a totally real number field. J. Nekov\'a\v{r} and A. Scholl \cite{NS16} formulated what they call the plectic conjecture. The geometric objects in this conjecture are Shimura varieties/stacks whose definition groups are restrictions of scalars from an algebraic group over $F$. More concretely, they work with abelian varieties with real multiplication by $O_{F}$, where $O_{F}$ is the ring of integers of $F$. The plectic principle includes, among others, Oda's conjecture about factorization of periods of Hilbert modular forms. However, we only use the weak version of equivariant cohomology to construct the plectic Green functions. We now explain the plectic principle of this version.

  Let $B$ be a connected complex manifold, $X /B $ a family of abelian varieties with real multiplication, and $s: B \longrightarrow X$ a nonzero torsion section fixed by a subgroup of finite index $U\subset O^{\times}_{F,+}$, which is the group of totally positive units. This subgroup $U$ acts naturally on $X$ and acts trivially on $B$, then we should consider the 
following diagram 
\begin{diagram}
X\times EU &\rTo   &X\times_{U}EU         &\rTo   &X\\
\dTo &       &  \dTo  &           &\dTo\\
B\times EU &\rTo    & B\times_{U}EU       &\rTo^{\pi}  &B,
\end{diagram}
where $EU$ is the topological total space over the classifying space $BU$ of the group $U$ and 
$$B\times_{U}EU = B\times \left ( EU/U\right), \quad X\times_{U}EU = (X\times EU)/U$$ 

Since $U$ is just a discrete group, then
$EU \cong U_{\mathbb{R}}$. 

  If $B = \{ pt \}$, then $X$ is a variety and we have the following situation 
\begin{diagram} \label{diag:plec}
X\times U_{\mathbb{R}} & \rTo &X\times_{U}U_{\mathbb{R}} &\rTo & X \\   
\dTo &   &\dTo &  &\dTo \\
pt\times U_{\mathbb{R}} &\rTo & pt\times U_{\mathbb{R}}/U & \rTo & pt.         
\end{diagram}
  
   Nekov\'a\v r and Scholl \cite{NS16} constructed in their work a $U$-equivariant current $\tilde{\theta}(\cdot , \cdot)$ on $\tilde{X} = X\times U_{\mathbb{R}}$. In fact $\tilde{\theta}(\cdot ,\cdot)$ is a plectic generalization of the (slightly modified) $\log | \theta (\tau, z)|$ of the absolute value of the standard Theta function on the elliptic curve $E = \mathbb{C}/(\mathbb{Z}\tau + \mathbb{Z})$, which is the Green function on $E$.
 
  So $s^{\ast}(\tilde{\theta})$ is $U$-equivariant on $\tilde{B} = B\times U_{\mathbb{R}}$. Then $s^{\ast}(\tilde{\theta})$ can descend to a current on $B\times (U_{\mathbb{R}}/U)$, and we can compute the trace 
\[ \pi_{\ast}(s^{\ast}(\tilde{\theta}))  = \int_{U_{\mathbb{R}}/U}s^{\ast}(\tilde{\theta}),\]
which gives very interesting functions, such as generalized Eisenstein-Kronecker-Lerch series. 

  The above integral, as well as its variants involving more complicated functions than
$\tilde\theta(\cdot ,\cdot)$, can be computed by integrating suitable expressions depending on
$\Vert ux\Vert$ over $U_{\mathbb{R}}$. For this purpose, the Hecke transform (see Definition-Proposition \ref{def:hec} in section 3) is 
introduced and used in their work. Here is a typical example. We will give a new interpretation of the Hecke formula in the spirit of the plectic principle.
  
  Combining the idea of Hodge correlators and the plectic principle, we construct the higher plectic Green functions (see Definition \ref{defi:1}) $G_{I, \nu,\Gamma, S}(\{x_{v}\}_{v\in S}, u)$ on $\left( F_{\mathbb{R}}/I\right)^{|S|}\times U_{\mathbb{R}}$. These functions depend on a fixed fractional ideal $I$ and some combinatorial data, namely a non-oriented graph $\Gamma$ and a finite subset $S$ of the set of all vertices.  
  
  The higher plectic Green functions constitute the key of our attempt to generalize the Hecke formula. 
  We can therefore define a multivariable function (see Definition \ref{defi:2}) as 
  \[ \mathscr{F}_{I, \nu,\Gamma, S}(\{x_{v}\}_{v\in S}) =   \frac{1}{[O^{\times}_{F,+}: U]}\int_{U_{\mathbb{R}}/U} G_{I, \nu, \Gamma, S}(\{x_{v}\}_{v\in S}, u)\ \mathrm{d}^{\times}u, \]
  and define the \lq\lq generalized multiple zeta value\rq\rq \ (see Definition \ref{defi:3}) as
  \[ Z_{I, \nu}(\Gamma, S) = \mathscr{F}_{I, \nu, \Gamma, S}(\{0\}_{v\in S}). \]
 
 \paragraph{\textbf{Main results}} After general construction, in this article we will focus on the rational field $\mathbb{Q}$ and show that 

\begin{theo} [See Theorem \ref{theo:1}]
If $F = \mathbb{Q}, I^{\ast}=\mathbb{Z}$, then the generalized multiple zeta value $ Z_{I, \nu}(\Gamma, S)$ associated to an arbitrary tree $\Gamma$ is a finite $\mathbb{Z}$-linear combination of classical MZVs of depth and weight determined by the given tree. 
 \end{theo}

\begin{theo} [See Theorem \ref{thm:poly}]
If $F = \mathbb{Q}, I^{\ast}=\mathbb{Z}$, $x_{v} \in \frac{1}{N}\mathbb{Z}/\mathbb{Z}$, then
$G_{I,\nu,\Gamma, \partial \Gamma}(\{x_{v}\}_{v\in \partial \Gamma}, 1)$ is a finite $\mathbb{Z}$-linear combination of the 
values of multiple polylogarithms of depth $d = \mathrm{rank}( H_1(\Gamma, S))$ evaluated at some $N$-th roots of unity.
\end{theo}

If the number field in consideration is an arbitrary number field of degree $d$, the generalized multiple zeta values involve highly non-trivial iterated integrals, which are higher dimensional generalizations of polylogarithms. Their properties will be discussed in a separate article.

% The plan of this paper. 

\subsection*{Notations and conventions}
\begin{altitemize}
 \item Let $F$ be a totally real field of degree $[F:\mathbb{Q}] = r$ and $O_{F}$ the ring of integers. Let $O^{\times}_{F,+}$ be the group of totally positive units and $U$ a subgroup of finite index of $O^{\times}_{F,+}$. Let
 \[U_{\mathbb{R}} = \mathrm{Ker}(N:(F^{\times}_{\mathbb{R}})_{+}\longrightarrow \mathbb{R}^{\times}_{+}) = \{(u_{1},\ldots,u_{r})\in \mathbb{R}^{r}_{+}\mid u_{1}\cdots u_{r} = 1\}.\]
 
 Let $I$ be an ideal of $F$. $\mathrm{Tr} = \mathrm{Tr}_{F/\mathbb{Q}}$ and $N = N_{F/\mathbb{Q}}$ denote respectively the trace and norm map. Let $\mathscr{D}$ be the different ideal and $I^{\ast}$ the dual ideal of $I$
 \[ I^{\ast} = \{ a\in F | \mathrm{Tr}(a I) \in \mathbb{Z}\} = \mathscr{D}^{-1}I^{-1}.\]
 
 Let $\mid\mid\cdot \mid\mid: F_{\mathbb{R}} = F\otimes \mathbb{R} \DistTo \mathbb{R}^{\mathrm{Hom}(F,\mathbb{R})}\longrightarrow \mathbb{R}_{+} \cup \{0\}$ be the standard euclidean norm. 
 
 \item Let $K$ be an arbitrary number field of degree $[K:\mathbb{Q}] = r_1 + 2r_2 = r$.  
 \item Let $C^a_b $ denote ${b}\choose {a}$ the binomial coefficient.
\end{altitemize}

\subsection*{Acknowledgments} The author thanks J. Nekov\'a\v r for introducing this interesting problem and for helpful discussions. This paper is part of the author's thesis of Sorbonne Universit\'e. The author thanks M. Kaneko and D. Zagier for careful reading and inspiring comments. The author also wants to thank Y. Manin for helpful discussions. This paper was written during the author's stay at the Max Planck Institute for Mathematics in Bonn, whose hospitality and financial support is greatly appreciated.

\section{Plectic Green functions}
Inspired by Goncharov's work, we look for an analogue of Green's function over number fields. In this section, we will focus on $F$. We need to consider a generalization of $\log |1- e^{2\pi ix}|^{2}$ on the compact real torus 
\[S^{1} = \{e^{2\pi ix} | x \in \mathbb{R}/\mathbb{Z}\}\subset \mathbb{C}^{\times}.\]
This function is the restriction of the Green function $G(1,y) = \log |1 - y|^2$ of the origin of $\mathbb{C}^\times$ to $S^1$.
We are going to consider corresponding objects on tori with real multiplication.

\subsection{Plectic Green functions}
Let us recall Diagram (\ref{diag:plec}) for plectic philosophy. In our case, $B$ is just a point and $X = F_{\mathbb{R}}/I$ the real torus with real multiplication by $O_{F}$.  The condition of a nonzero torsion section $s: B \longrightarrow X$ fixed by a subgroup of finite index $U\subset O^{\times}_{F,+}$ in the plectic picture implies that $x\in F_{\mathbb{R}}/I$ is a torsion point. Such a subgroup $U$ exists if and only if $x$ lies in the torsion subgroup of $F_{\mathbb{R}}/I$. 
The picture for plectic principle now turns out to be
\begin{diagram} 
(F_{\mathbb{R}}/I)\times U_{\mathbb{R}} & \rTo &(F_{\mathbb{R}}/I)\times_{U}U_{\mathbb{R}} &\rTo & F_{\mathbb{R}}/I \\   
\dTo &   &\dTo &  &\dTo \\
pt\times U_{\mathbb{R}} &\rTo & pt\times (U_{\mathbb{R}}/U) & \rTo & pt.         
\end{diagram}

  In order to apply the plectic principle, we will firstly construct objects on $\left(F_{\mathbb{R}}/I\right)\times U_{\mathbb{R}}$, namely plectic Green functions.
  \begin{defi}[\textbf{Plectic Green function}] \label{defi:plecticGreen} 

  The plectic Green function associated to the ideal $I$ is defined as
\[ g_{I}(x,u) = \lim_{\eta \rightarrow 0^{+}} \sum_{n\in I^{\ast}\backslash \{0\}}\frac{e^{2\pi i\mathrm{Tr}(nx)}}{||un||^{r+\eta}}, \quad  x\in F_{\mathbb{R}}/I, \ u\in U_{\mathbb{R}}.\]
This function can be viewed as a distribution on $(F_{\mathbb{R}}/I)\times U_{\mathbb{R}}$. 
  \end{defi}

  \subsubsection*{\textbf{Modified plectic Green functions}}
  Let $J_{F} = Hom(F, \mathbb{R})$ be the set of all field embeddings 
  \[ \{ \varsigma :F \hookrightarrow \mathbb{R} \}.\] 
  We can modify the plectic Green function by adding an additional choice of multisigns 
  $$\nu : J_{F} \longrightarrow \{0,1\}.$$
  Let us denote 
  \[\mathrm{sgn}(n)^{\nu} = \prod_{\varsigma \in J_{F}}\left(\mathrm{sgn}(\varsigma(n)) \right)^{\nu(\varsigma)}, \]
  % \[\mathrm{sgn}(n) = \prod_{\varsigma \in J_{F}}\left(\mathrm{sgn}(\varsigma(n)) \right). \]
  and we make a convention of notation
  \[ \mathrm{sgn}(n) = (-1)^{\nu} \Longleftrightarrow \ \mathrm{sgn}(\varsigma(n)) = (-1)^{\nu(\varsigma)}, \qquad  \forall \varsigma \in J_{F}.\]
  There are two ways of modification of the plectic Green function, one is defined as
\[ g^{\nu}_{I}(x,u) = \lim_{\eta \rightarrow 0^{+}} \sum_{n\in I^{\ast}\backslash \{0\}}\mathrm{sgn}(n)^{\nu}\frac{e^{2\pi i\mathrm{Tr}(nx)}}{||un||^{r+\eta}}, \qquad  x\in F_{\mathbb{R}}/I, \ u\in U_{\mathbb{R}},\]
and another is defined as
\[ g_{I,\nu}(x,u) = \lim_{\eta \rightarrow 0^{+}} \sum_{\substack{n\in I^{\ast}\backslash \{0\} \\ \mathrm{sgn}(n) = (-1)^{\nu}}}\frac{e^{2\pi i\mathrm{Tr}(nx)}}{||un||^{r+\eta}}, \qquad  x\in F_{\mathbb{R}}/I, \ u\in U_{\mathbb{R}}.\]

% \subsubsection*{\textbf{Properties of plectic Green functions}} 
  In order to descend the plectic Green function to $ U_{\mathbb{R}}/U $ and then to compute the trace, we need to verify that the plectic Green function is 
  $U$-equivariant. On the other hand, it is natural to ask how our definition depends on the choice of ideal $I$. The following lemmas answer these questions. 
 \begin{lem}[$O^{\times}_{F,+}-$equivariance]
$\forall \epsilon \in O^{\times}_{F,+}$, we have 
\[ g_{I}(\epsilon x, \epsilon u) = g_{I}(x,u).\]
This is also true for $g^{\nu}_{I}(\cdot,\cdot)$ and $g_{I,\nu}(\cdot,\cdot)$.
\end{lem}
\begin{lem}[Dependence on $I$]%verify)
1. If $\alpha \in F^{\times}_{+}$ and $N(\alpha) = 1$, then $(\alpha I)^{\ast} = \alpha ^{-1}I^{\ast}$. Hence
\[ g_{\alpha I} (\alpha x, \alpha u) = g_{I}(x , u).\]
2.  For any $\alpha \in F^{\times}_{+}$, $g_{\alpha I}(\alpha x, u) = N(\alpha) g_{I}(x,u)$. Therefore up to rescaling, $g_{I}(\cdot, \cdot)$ depends only on the 
class of $I$ in the class group $ Cl^{+}_{F}$.
This is also true for $g^{\nu}_{I}(\cdot, \cdot)$ and $g_{I,\nu}(\cdot,\cdot)$.
\end{lem}

\subsection{Higher plectic Green functions}
   As explained in the introduction, we will use Goncharov's idea for Hodge correlators to construct new objects with multiple variables. More precisely, we will use the plectic Green functions as fundamental block to construct an object $G_{I, \Gamma, S}(\cdot , \cdot)$ on 
\[O^{\times}_{F,+}\backslash (F_{\mathbb{R}}/I)^{S}\times U_{\mathbb{R}} = (F_{\mathbb{R}}/I)^{S}\times_{O^{\times}_{F,+}}EO^{\times}_{F,+},\]
which depends on a graph $\Gamma $ and a subset $S$ of the set of its vertices.  

\begin{defi}(\textbf{Higher plectic Green function})\label{defi:1} \\
Let $\Gamma$ be a finite connected non-oriented graph, $V(\Gamma)$ the set of vertices and $E(\Gamma)$ the non-empty set of edges.
Let $S\subset V(\Gamma) $ be a subset of the set of vertices. Loops are forbidden here (i.e. the endpoints of each edge are distinct), but multiple edges are allowed. For each vertex $v\in V(\Gamma)$, let $x_{v}\in F_{\mathbb{R}}/I$ be a variable which decorates the vertex $v$; for each edge $e\in E(\Gamma)$, 
we fix an orientation $\overrightarrow{e} = (v_{0}(e)\longrightarrow v_{1}(e))$ here we denote $v_{1}(e)$ (resp. $v_{0}(e)$) the head (resp. the tail) of the arrow.
We associate an element $n_{e} \in I^{\ast}\setminus\{0\}$ to the edge $e$. Then for each edge $e$, we can associate a plectic Green function
\[g_{I}(x_{v_{1}(e)}-x_{v_{0}(e)},u) = \lim_{\eta \rightarrow 0^{+}} \sum_{n_e\in I^{\ast}\backslash \{0\}}\frac{e^{2\pi i\mathrm{Tr}(n_e(x_{v_{1}(e)}-x_{v_{0}(e)}))}}{||un_e||^{r+\eta}}.\]

  We define the \textbf{higher plectic Green function} attached to $(\Gamma, S)$ as
\[ G_{I, \Gamma, S}(\{x_{v}\}_{v\in S}, u) = \int_{(F_{\mathbb{R}}/I)^{|V(\Gamma)\smallsetminus S|}}\prod_{e\in E(\Gamma)}g_{I}(x_{v_{1}(e)}-x_{v_{0}(e)},u)\prod_{v\in V(\Gamma)\smallsetminus S}dx_{v}, \]
where $x_{v}\in F_{\mathbb{R}}/I, u\in U_{\mathbb{R}}$ and $dx$ is a fixed Haar measure en $F_{\mathbb{R}}$. 
\end{defi}
  
  Roughly speaking, the higher plectic Green function is defined by integration of the product of the basic plectic Green function associated to each edge respect to all the variables decorating the vertex $v \in V(\Gamma) \setminus S$. We should also mention that the higher plectic Green function does not depend on the orientation that we fix for each edge $e$.

\begin{re}\label{re:tor}
We are going to consider only the values of $\{x_v\}_{v\in S}$ lying in the torsion group of $F_{\mathbb{R}}/I$. This is equivalent to the existence of a subgroup of finite index $U \subset O^{\times}_{F,+}$ fixing each $x_v$.
\end{re}

  There are variants of these functions depending
on an additional choice of multisigns $\nu_{e}: J_{F} \longrightarrow \{0, 1\}$ (and an orientation)
for each edge $e$, which means that we can replace $g_{I}(\cdot, \cdot)$ by $g_{I,\nu}(\cdot,\cdot)$ (or $g^{\nu}_{I}(\cdot,\cdot)$) in the definition of the higher
plectic Green function. 

\begin{re} 
By the very definition, $G_{I,\Gamma, S}(\cdot, \cdot)$ inherits a $O^{\times}_{F,+}$-invariant property. $\forall \epsilon \in O^{\times}_{F,+}$, we have 
\[G_{I,\Gamma, S}(\{\epsilon x_{v}\}_{v\in S},\epsilon u) = G_{I,\Gamma,S}(\{x_{v}\}_{v\in S}, u).\]
Therefore, our higher plectic Green function $G_{I,\Gamma,S}(\cdot,\cdot)$ is indeed defined on 
\[O^{\times}_{F,+}\backslash (F_{\mathbb{R}}/I)^{S}\times U_{\mathbb{R}} = (F_{\mathbb{R}}/I)^{S}\times_{O^{\times}_{F,+}}EO^{\times}_{F,+},\]
which depends on the given graph $\Gamma $ and the subset $S$ of the set of its vertices.  
Here $EO^{\times}_{F,+}$ is the total space of the group $O^{\times}_{F,+}$. We can therefore apply the plectic principle and later we will compute the trace 
\[\int_{U_{\mathbb{R}}/U}G_{I,\Gamma, S}(\{x_v\}_{v\in S}, u)\ \mathrm{d}^{\times}u,\]
when $\{x_v\}_{v\in S}$ and $U$ are as in Remark \ref{re:tor} (see the discussion of Hecke's formula in Section 3).
\end{re}

\subsubsection{Fourier expansion of higher plectic Green functions}
Higher plectic Green functions are defined by integration, however their Fourier expansions are just series. For achieving this, we need the following lemma

 \begin{lem}(\textbf{Convolution on $F_{\mathbb{R}}/I$}) \label{lem:formalconvolution}
   
    Let $\chi_{n}(x) = e ^{2\pi i\mathrm{Tr}(xn)}$.  
 If $A(x) = \sum_{m\in I^{\ast}}a(m)\chi_{m}(x)$ and $B(y) = \sum_{n\in I^{\ast}}b(n)\chi_{n}(y)$, then
\[ \int_{F_{\mathbb{R}}/I}A(x-y)B(y)dy = \int_{F_{\mathbb{R}}/I}\sum_{m,n\in I^{\ast}}a(m)b(n)\chi_{m}(x-y)\chi_{n}(y)dy \]
\[= vol(F_{\mathbb{R}}/I)\sum_{n\in I^{\ast}}a(n)b(n)\chi(n).\]
\end{lem}

 \begin{proof}
The Proof of Lemma \ref{lem:formalconvolution} is straightforward. 
\[ \int_{F_{\mathbb{R}}/I}A(x-y)B(y)dy \]
\[=  \int_{F_{\mathbb{R}}/I}\sum_{m,n\in I^{\ast}}a(m)b(n)e ^{2\pi i\mathrm{Tr}((x-y)m)}e ^{2\pi i\mathrm{Tr}(yn)}dy \]
\[= \int_{F_{\mathbb{R}}/I}\sum_{m = n}a(m)b(n)e ^{2\pi i\mathrm{Tr}(xm)}e ^{2\pi i\mathrm{Tr}(y(n-m))}dy + \int_{F_{\mathbb{R}}/I}\sum_{m \neq n}a(m)b(n)e ^{2\pi i\mathrm{Tr}(xm)}e ^{2\pi i\mathrm{Tr}(y(n-m))}dy.\]
If $ m \neq n $, then
\[\int_{F_{\mathbb{R}}/I}b(n)e ^{2\pi i\mathrm{Tr}(y(n-m))}dy = 0.\]
Therefore we obtain
\[\int_{F_{\mathbb{R}}/I}A(x-y)B(y)dy  =  \int_{F_{\mathbb{R}}/I}\sum_{m = n}a(m)b(n)e ^{2\pi i\mathrm{Tr}(xm)}e ^{2\pi i\mathrm{Tr}(y(n-m))}dy  \] 
\[ = \int_{F_{\mathbb{R}}/I}\sum_{n \in I^{\ast}}a(n)b(n)e ^{2\pi i\mathrm{Tr}(xn)}\cdot \int_{F_{\mathbb{R}}/I}1\ dy = vol(F_{\mathbb{R}}/I)\sum_{n\in I^{\ast}}a(n)b(n)e ^{2\pi i\mathrm{Tr}(xn)}. \]
 \end{proof}

 We will now apply Lemma (\ref{lem:formalconvolution}) to higher plectic Green function. We put
\begin{align*}
 n : E \longrightarrow I^{\ast}\setminus \{0\} ; \quad e \longmapsto n_{e} .
\end{align*}

  By the definition of the higher plectic Green function we have
\begin{align*}\label{int:1}
& G_{I,\Gamma,S}(\{x_{v}\}_{v\in S}, u)  = \\
\lim _{\eta \rightarrow 0^{+}}\sum_{n:E(\Gamma)\rightarrow I^{\ast}\setminus\{0\}}
\prod_{e\in E(\Gamma)}||un_{e}||^{-r-\eta} & \int_{(F_{\mathbb{R}}/I)^{V(\Gamma)\smallsetminus S}}e ^{2\pi i\mathrm{Tr}(\sum_{e\in E(\Gamma)}n_{e}(x_{v_{1}(e)}-x_{v_{0}(e)}))} \prod_{v\in V(\Gamma)\smallsetminus S}dx_{v},
\end{align*}

   If $\mathbb{Z}[X]$ denotes the free abelian group on a set X. We define the chain complex for the graph $\Gamma$ as
\[ \delta : C_1(\Gamma) = \mathbb{Z}[E(\Gamma)] \longrightarrow C_0(\Gamma) = \mathbb{Z}[V(\Gamma)],\]
where $\delta : (v_0 \rightarrow v_1) \longmapsto [v_1]-[v_0] $ is the boundary map of the chain complex. 

  We can also define the relative chain complex for $(\Gamma,S)$, namely, $$C_1(\Gamma) \longrightarrow C_0(\Gamma)/C_0(S).$$
We can associate to $n$ the following element $c(n)$ in $C_{1}(\Gamma)\otimes_{\mathbb{Z}}I^{\ast}$ of the graph $\Gamma$.
\[ c(n) = \sum_{e\in E(\Gamma)} n_{e}\cdot \overrightarrow{e} \in C_{1}(\Gamma)\otimes_{\mathbb{Z}}I^{\ast}.\]
If let 
\begin{equation}\label{const:1}
\pi_{v} = \sum_{e\in E(\Gamma),v_{1}(e)=v}n_{e} - \sum_{e\in E(\Gamma),v_{0}(e)=v }n_{e},
\end{equation}
 then 
 \[e ^{2\pi i \mathrm{Tr}(\sum_{e\in E(\Gamma)}n_{e}(x_{v_{1}(e)}-x_{v_{0}(e)}))} = e ^{2\pi i\mathrm{Tr}(\sum_{v\in V(\Gamma)}\pi_{v}x_{v})}. \]
We define the boundary map
$$\partial n : V(\Gamma) \longrightarrow I^{\ast},$$ 
$$ \partial n (v) = \delta c(n) |_{v} = \pi_{v},$$
where $|_{v}$ means taking the coefficient of the vertex $v$.

  By using the previous convolution formula, we conclude that only the terms with $\partial n |_{V(\Gamma)\smallsetminus S} = 0 $ contribute to the integral in Definition (\ref{defi:1}), which means  that
$$ \forall v\in V(\Gamma)\setminus S, \quad \pi_{v} =0 .$$ 

  Note that
\[\{ c(n) \mid \partial n |_{V(\Gamma)\smallsetminus S} = 0 \} = H_{1}(\Gamma, S)\otimes_{\mathbb{Z}}I^{\ast}.\]
Then we get a formal Fourier convolution description of $G_{I,\Gamma,S}(\cdot, \cdot)$. 
  \begin{prop}(\textbf{Fourier Expansion})
$$G_{I,\Gamma,S}(\{x_{v}\}_{v\in S}, u) = vol(F_{\mathbb{R}}/I)^{|V(\Gamma)\smallsetminus S|}\lim _{\eta \rightarrow 0^{+}}\sum\nolimits '_{\{n, c(n) \in H_{1}(\Gamma,S)\otimes I^{\ast}\}}
\frac{e ^{2\pi i\mathrm{Tr}(\sum_{v\in S} (\partial n)_{v}x_{v})}}{\prod_{e\in E(\Gamma)}||un_{e}||^{r+\eta}},$$
where $\sum'$ means that we consider only $n$ such that $c(n)\in H_{1}(\Gamma,S)\otimes I^{\ast}$ and 
$$\forall e \in E(\Gamma), n_{e} \in I^{\ast}\setminus \{0\}.$$
  \end{prop}

 \begin{proof}
 The proof is straightforward by applying Lemma (\ref{lem:formalconvolution}). 
 \end{proof}

\subsubsection{Subdivision operation}

 We can replace each $e \in E(\Gamma)$ by a chain of $k_{e} \geq 1$ edges, which is equivalent to that 
we add $k_{e}-1 \geq 0$ new vertices to each edge $e\in E(\Gamma )$ to get a new graph $\Gamma(\underline{k})$ with $|V(\Gamma(\underline{k}))| = |V(\Gamma)| + \sum_{e}(k_{e}-1)$ and the subset $S$ is unchanged. 

For example, the case of $k_e = 3$ is as follows. 
\begin{figure}[h]
 \centering
    \includegraphics[width=0.5\textwidth]{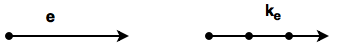}
    \caption{The subdivision of the edge $e$.}
    \label{fig:e}
\end{figure}

Hence we define a subdivision map: 
\begin{defi}(\textbf{Subdivision map})\label{defi:subdiv}
\[\underline{k}: E(\Gamma) \longrightarrow \mathbb{N}\setminus \{0\}\]
\[\underline{k} : e \longmapsto k_{e}\]
\end{defi}
Combining a subdivision of edges and the Fourier expansion, we get 
\begin{prop} (\textbf{Subdivision of edges})\label{prop:2}
$$ G_{I, \Gamma(\underline{k}), S}(\{x_{v}\}_{v\in S}, u) =  $$
$$ \mathrm{vol}(F_{\mathbb{R}}/I)^{|V(\Gamma)\backslash S|+|\underline{k}|-|E(\Gamma)|}\lim_{\eta \rightarrow 0^{+}} \sum_{\{n, c(n) \in H_{1}(\Gamma,S)\otimes I^{\ast}\}}\nolimits '\frac{e^{2\pi i\mathrm{Tr}(\sum_{v\in S}(\partial n)_{v}x_{v})}}{\prod_{e\in E(\Gamma)}||un_{e}||^{k_{e}(r+\eta)}},  $$
where $ x\in F_{\mathbb{R}}/\mathbb{R}, u\in U_{\mathbb{R}}$, $|\underline{k}| = \sum_{e\in E(\Gamma)} |k_{e} |$, $n = (n_e)_{e\in E(\Gamma)}$,
$$\partial n : V(\Gamma) \longrightarrow I^{\ast},$$ 
$$ \partial n (v) = \sum_{e\in E(\Gamma),v_{1}(e)=v}n_{e} - \sum_{e\in E(\Gamma),v_{0}(e)=v }n_{e}, $$
and $\sum '$ means that $\partial n$ is supported at $S$ and each $n_e$ is nonzero. 
\end{prop}
\begin{re} \label{sign}
We can also add multisigns here as we did for plectic Green functions before. The modified higher plectic Green function can be defined as
\[G^{\nu}_{I, \Gamma(\underline{k}), S}(\{x_{v}\}_{v \in S}, u) = \]
\[ vol(F_{\mathbb{R}}/I)^{|V(\Gamma)\backslash S|+|\underline{k}|-|E(\Gamma)|}\lim_{\eta \rightarrow 0^{+}}\sum_{\{n, c(n) \in H_{1}(\Gamma,S)\otimes I^{\ast}\}}\nolimits '\mathrm{sgn}(n_e)^{\nu(e)}\frac{e^{2\pi i\mathrm{Tr}(\sum_{v\in S}(\partial n)_{v}x_{v})}}{\prod_{e\in E(\Gamma)}||un_{e}||^{k_{e}(r+\eta)}}, \]
or as
$$
G_{I, \nu,\Gamma(\underline{k}), S}(\{x_{v}\}_{v \in S}, u) = $$ 
$$ vol(F_{\mathbb{R}}/I)^{|V(\Gamma)\backslash S|+|\underline{k}|-|E(\Gamma)|}\lim_{\eta \rightarrow 0^{+}} \sum_{\substack{\{n, c(n) \in H_{1}(\Gamma,S)\otimes I^{\ast}\} \\ \mathrm{sgn}(n_e) = (-1)^{\nu(e)}}}\nolimits '\frac{e^{2\pi i\mathrm{Tr}(\sum_{v\in S}(\partial n)_{v}x_{v})}}{\prod_{e\in E(\Gamma)}||un_{e}||^{k_{e}(r+\eta)}}.
$$
\end{re}

\begin{re}
We will only consider the graphs whose internal vertices' valency is no smaller than $3$, due to the subdivision map.
\end{re}

\section{The generalized multiple zeta values}

\subsection{The generalized multiple zeta values} 
  
\begin{defi}\label{defi:2}
    We define a new multivariable function associated to the ideal $I$ and the combinatorial data $(\Gamma, S)$ as follows
\[ \mathscr{F}_{I,\Gamma, S}(\{x_{v}\}_{v\in S}) = (O^{\times}_{F,+}:U)^{-1}\int_{U_{\mathbb{R}}/U}G_{I,\Gamma,S}(\{x_{v}\}_{v\in S}, u )\mathrm{d}^{\times}u,\]
where $U\subset O^{\times}_{F,+}$ is a subgroup of finite index and $x_v \in \left(F_{\mathbb{R}}/I\right)^{U}$ for all $v\in S$, 
\[U_{\mathbb{R}} = \{u = (u_{1},\ldots ,u_{r}) \in (\mathbb{R}^{\times}_{+})^{r} \vert \quad \prod^{r}_{j=1} u_{j} = 1 \},\]
$$\mathrm{d}^{\times}u = \frac{du_{1}\cdots du_{r-1}}{u_{1}\cdots u_{r-1}}, $$ 
and $U_{\mathbb{R}}/U \cong (\mathbb{S})^{r-1}$ is the classifying space of $U \cong \mathbb{Z}^{r-1}$. 
  
  In the same way, we can define $\mathscr{F}^{\nu}_{I,\Gamma, S}(\{x_{v}\}_{v\in S}) $ and $\mathscr{F}_{I,\nu,\Gamma, S}(\{x_{v}\}_{v\in S}) $.
\end{defi}

\begin{re} 

\end{re}

\begin{defi}(\textbf{Generalized Multiple Zeta Values}) \label{defi:3}
   
   The generalized multiple zeta value is defined as 
\[Z_{I}(\Gamma, S) = \mathscr{F}_{I,\Gamma,S}(\{0\}_{v\in S}).\]

Similarily, we have 
\[Z^{\nu}_{I}(\Gamma, S) = \mathscr{F}^{\nu}_{I,\Gamma,S}(\{0\}_{v\in S}),\]
and 
\[Z_{I,\nu}(\Gamma, S) = \mathscr{F}_{I,\nu,\Gamma,S}(\{0\}_{v\in S}).\]
\end{defi}

%\begin{re}
%In fact, the construction for $G_{I, \nu,\Gamma(\underline{k}), S}(\{x_{v}\}_{v \in S}, u)$ also works for arbitrary number fields. If a complex place
%exists, we can just replace the Euclidean norm by Hermitian norm and we can add signs just for the real places. For an imaginary quadratic field, there is no Hecke transform. 
%\end{re} 
\subsection{The Hecke transform}

In the integral $\int_{U_{\mathbb{R}}/U}$, the $U$-invariance of $G_{I,\nu,\Gamma,S}(\{x_{v}\}_{v\in S}, u )$ enables us to firstly consider the integral 
$\int_{U_{\mathbb{R}}}$. This calculation necessitates the Hecke transform.

\begin{defi-prop}[\textbf{The Hecke transform \cite{NS}}] \label{def:hec}
Let $ U_{\mathbb{R}} \subset (\mathbb{R}^{\times}_{+})^{r}$ be the subgroup 
\[ U_{\mathbb{R}} = \{u = (u_{1},\ldots ,u_{r}) \in (\mathbb{R}^{\times}_{+})^{r} \vert \prod^{r}_{j=1} u_{j} = 1 \}. \]
Let $\Vert \cdot \Vert $ be the 
Euclidean norm on $\mathbb{C}^{r}$, on which $U_{\mathbb{R}}$ acts by multiplication. Let $(p)_{j}\in \mathbb{Z}^{r}$, $ p = \sum p_{j}$. Then
for any $x = (x_1, \ldots, x_r)\in (\mathbb{C}^{\times})^{r}$ and $s \in \mathbb{C}$, $Re(s) > 0 $, \textbf{the} \textbf{Hecke} \textbf{transform} is 
\[ \int_{U_{\mathbb{R}}} \vert | ux \vert | ^{-2s}\prod_{j} u_{j}^{-2p_{j}}\ \mathrm{d}^{\times}u = \frac{2^{1-r}}{r\Gamma(s)}\prod_{j} \Gamma(\frac{p+s}{r} - p_{j})\vert x_{j}\vert ^{2(p_{j}-(p+s)/r)}.\]
\end{defi-prop}
\begin{proof}
See \cite{NS}.
\end{proof}

%\begin{re}
%This formula works also for any $(a_{1},\cdots, a_{n}) \in \mathbb{C}^{n}$ and $s\in \mathbb{C}$ such that $\forall j = 1, \cdots, n$, $2Re(s) + na_{j} > a_{1} + \cdots + a_{n}$.
%\end{re}
Let us see one special example, namely the integral of the plectic Green function 
$$g_{I}^{\nu}(x,u) = \lim_{\delta \rightarrow 0^{+}}\sum_{n\in I^{\ast}\smallsetminus\{0\}}\mathrm{sgn}(n)^{\nu}\frac{e^{2\pi i\mathrm{Tr}(nx)}}{\Vert un \Vert^{r+\delta}}, \quad x \in F_{\mathbb{R}}/I,  u \in U_{\mathbb{R}} .$$

The Hecke transform of the plectic Green function is behind the proof of Hecke's formula, as we are now going to explain.
\begin{theo}[\textbf{New interpretation of the Hecke formula}] \label{thm:hecke}

Let us suppose that $x \in F_{\mathbb{R}}/I$ is a torsion element, then there exists a subgroup $U$ of $\subset O^{\times}_{F,+}$ of finite index such that
$$ x \in (F_{\mathbb{R}}/I)^{U},$$
then 
\[\forall \epsilon \in U, \quad g_{I}^{\nu}(x,\epsilon u) = g_{I}^{\nu}(x,u).\]
Then we obtain
\[\int_{U_{\mathbb{R}}/U}g_{I}^{\nu}(x,u)\mathrm{d}^{\times} u = \frac{2^{1-r}\Gamma(1/2)^{r}}{r\Gamma(r/2)}\lim_{\delta\rightarrow 0^{+}}\sum_{n\in (I^{\ast}\smallsetminus\{0\})/U}\mathrm{sgn}(n)^{\nu}\frac{e^{2\pi i\mathrm{Tr}(nx)}}{\vert N(n)\vert^{(r+\delta)/r}}.\]
\end{theo}
\begin{proof}
By the definition
\[\int_{U_{\mathbb{R}}/U}g_{I}^{\nu}(x,u)\mathrm{d}^{\times} u  = \lim_{\delta \rightarrow 0^{+}}\sum_{n\in (I^{\ast}\smallsetminus\{0\})/U}e^{2\pi i\mathrm{Tr}(nx)}\int_{U_{\mathbb{R}}}\frac{1}{\Vert un\Vert^{r+\delta}}\mathrm{d}^{\times} u .  \]
By the Hecke transform, we have 
\[\int_{U_{\mathbb{R}}}\frac{1}{\Vert un\Vert^{r+\delta}}\mathrm{d}^{\times} u = \frac{2^{1-r}\left(\Gamma\left(\frac{r+\delta}{2r}\right) \right)^{r}}{r\Gamma\left(\frac{r+\delta}{2}\right)}\frac{1}{\prod^{r}_{j=1}\vert n_{j}\vert^{(r+\delta)/r}}\]
then we obtain
\[\int_{U_{\mathbb{R}}/U}g_{I}^{\nu}(x,u)\mathrm{d}^{\times} u  = \frac{2^{1-r}\Gamma(1/2)^{r}}{r\Gamma(r/2)}\lim_{\delta\rightarrow 0^{+}}\sum_{n\in (I^{\ast}\smallsetminus\{0\}/U)}\mathrm{sgn}(n)^{\nu}\frac{e^{2\pi i\mathrm{Tr}(nx)}}{\prod^{r}_{j=1}\vert n_{j}\vert^{(r+\delta)/r}} .\]
Note that
\[N(n) = N_{F/\mathbb{Q}}(n) = \prod^{r}_{j=1}n_{j},\]
hence 
\[\int_{U_{\mathbb{R}}/U}g_{I}^{\nu}(x,u)du = \frac{2^{1-r}\Gamma(1/2)^{r}}{r\Gamma(r/2)}\lim_{\delta\rightarrow 0^{+}}\sum_{n\in (I^{\ast}\smallsetminus\{0\})/U}\mathrm{sgn}(n)^{\nu}\frac{e^{2\pi i\mathrm{Tr}(nx)}}{\vert N(n)\vert^{(r+\delta)/r}}.\]
\end{proof}
It is easy to see that the Hecke transform of the basic plectic Green function with signature delivers a linear combination of special values $L(1,\chi_{F})$ for certain Dirichlet characters $\chi_{F}$ of $F$ of signature $\nu$.  The Hecke transform can give a natural reinterpretation of Hecke's formula. That is why such
a formula is called by Nekov\'a\v r and Scholl the Hecke transform \cite{NS}.

%remark: we should mention the work of M. Nori on the Eisenstein cohomology classes for the integral unimodular group. A rationality result (due to Klingen and Siegel) of the L-values. 

\vspace{1cm}
\section{Relation to classical MZVs}

  In this section, we focus on the case $F= \mathbb{Q}$ and show how we recover the classical objects.
\begin{theo}[\textbf{Relation to multiple zeta values}]\label{theo:1}
 Let $F $ be  the rational field $ \mathbb{Q}$ and the ideal $I$ in consideration is $\mathbb{Z}$.
Let $\Gamma$ be any tree with rank $d = \mathrm{rank}( H_1(\Gamma, S))\geq 2$, where $S= \partial \Gamma$. Assume that we are given a "sign" map  $\nu : E(\Gamma) \longrightarrow \{ 0, 1\}$ and a subdivision map $\underline{k}: E(\Gamma) \longrightarrow \mathbb{N} \setminus \{0\}$ as in Remark \ref{sign} and in Definition \ref{defi:subdiv}, respectively. Then the generalized multiple zeta value $Z_{I,\nu}(\Gamma(\underline{k}), \partial \Gamma(\underline{k})) $ can be expressed as a finite $\mathbb{Z}$-linear combination of classical multiple zeta values (MZVs) of depth $d$ and weight $ |\underline{k}| =\sum\limits_{e}k_e$.
\end{theo}

\begin{theo}(\textbf{Relation to multiple polylogarithms})\label{thm:poly}\\
If $F = \mathbb{Q}, I^{\ast}=\mathbb{Z}$, $x_{v} \in \frac{1}{N}\mathbb{Z}/\mathbb{Z}$, then
$G_{I,\nu,\Gamma, \partial \Gamma}(\{x_{v}\}_{v\in \partial \Gamma}, 1)$ is a finite $\mathbb{Z}$-linear combination of the 
values of multiple polylogarithms of depth $d = \mathrm{rank}( H_1(\Gamma, S))$ evaluated at some $N$-th roots of unity.
\end{theo}

\begin{re}
 Before proving this result in general we consider first the case when $\Gamma$ is a plane trivalent tree, 
i.e., a tree whose internal vertices are of valency $3$. A plane trivalent tree $\Gamma$ with $m+1$ external vertices has $2m-1$ edges and $m-1$ internal vertices. The rank of $\Gamma$ is $d = \mathrm{rank}( H_1(\Gamma, \partial \Gamma)) = m$. 

A tree is a connected graph with no loops, no multiple edges and $H_1(\Gamma, \mathbb{Z}) =0$.
\end{re}

\begin{figure}[h]
 \centering
    \includegraphics[width=1.0\textwidth]{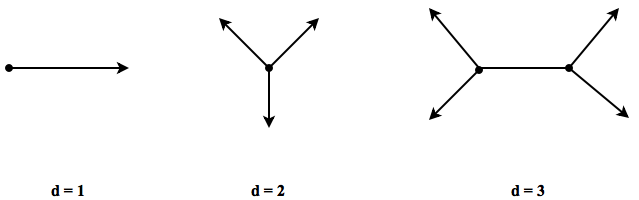}
    \caption{The simplest examples of plane trivalent trees of small ranks.}
    \label{fig:0}
\end{figure}

\begin{re}
In this theorem we only consider the situation that the subset $S = \partial \Gamma$. We would like to explain that such choice is quite general. In fact, $Z_{I,\nu}(\Gamma, S) = 0$ if  $\partial \Gamma \not\subset S$, therefore we must have $\partial \Gamma \subset S$. If one internal vertex is contained in $S$, namely $\partial \Gamma \subsetneqq S$, the situation can be reduced to a new tree with fewer external vertices due to the formal convolution Lemma (\ref{lem:formalconvolution}) 
\end{re}
The Hecke transform is trivial over the rational field. The generalized multiple zeta values are given 
by 
\[Z_{I,\nu}(\Gamma, S) = G_{I,\nu,\Gamma,S}(\{0\}_{v\in S}, 1).\] 

\subsection{Examples}
Before giving the proof in full generality, we will illustrate the statement of Theorem \ref{theo:1} by several examples. 

\begin{re}
In Theorem \ref{theo:1} and \ref{thm:poly}, we will take into account all arbitary trees, and we prove the theorems when the valency of any internal vertex $val(v_{internal}) \geq 3$. We will now explain the case: $val(v_{internal}) = 2$.

 Let $\widetilde{\Gamma}$ be one of the trees in Figure \ref{Fig}. 
\begin{figure}[h]
 \centering
    \includegraphics[width=1.0\textwidth]{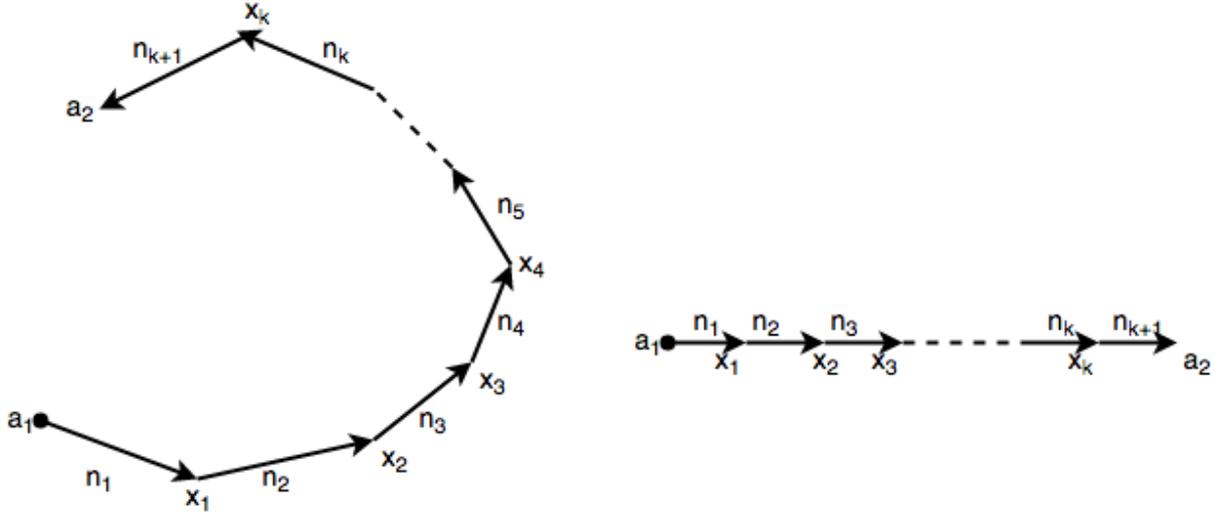}
    \caption{All internal vertices have valency 2}
    \label{Fig}
\end{figure}

We fix the signature $\nu_1 = 0$ for $n_1$, then implicitly all other signatures are also determined. Because of the formal convolution, we have 
$n_{i+1} -n_{i} =0$ for each internal vertex $x_{i}$ ($1 \leq i \leq k$).
\begin{align*}
&G_{\mathbb{Z},\nu,\widetilde{\Gamma}, \partial \widetilde{\Gamma}}(\{a_1,a_2\},1) \\
=& \int_{(\mathbb{R}/\mathbb{Z})^{k}}\sum_{\substack{n_1,\cdots,n_{k+1}\in \mathbb{Z}\smallsetminus\{0\}\\\mathrm{sgn}(n_j) = (-1)^{\nu_{j}}\\n_j= n_{j+1}, 1 \leq j \leq k}}\frac{e^{2\pi i \left(\sum\limits_{2\leq j \leq k}n_j(x_{j}-x_{j-1})) + n_1(x_1 - a_1) + n_{k+1}(a_2 - x_k)\right)}}{\prod^{k+1}_{j=1}\vert n _j \vert}\ dx_1\cdots dx_k\\
=& \sum_{n_1 \in \mathbb{N}\smallsetminus\{0\}} \frac{e^{2\pi in_1 (a_2-a_1)}}{n_1^{k+1}}.
\end{align*}

\[Z_{\mathbb{Z},\nu}(\widetilde{\Gamma}, \partial \widetilde{\Gamma}) = G_{\mathbb{Z},\nu,\widetilde{\Gamma}, \partial \widetilde{\Gamma}}(\{0,0\},1) = \sum_{n_1 \in \mathbb{N}\smallsetminus\{0\}} \frac{1}{n_1^{k+1}} = \zeta (k+1).\]
This result tells us that such a graph delivers the same result as the subdivision map of adding $k$ points for the tree with one edge and two external vertices $a_1, a_2$.
\end{re}

\paragraph{Non-tree case.}

We will show some examples for graphs which are not trees.
\begin{ex} \label{ex:nottree}
Let $\widehat{\Gamma}$ be the graph in Figure \ref{fig:a}. 

\begin{figure}[h]
 \centering
    \includegraphics[width=0.5\textwidth]{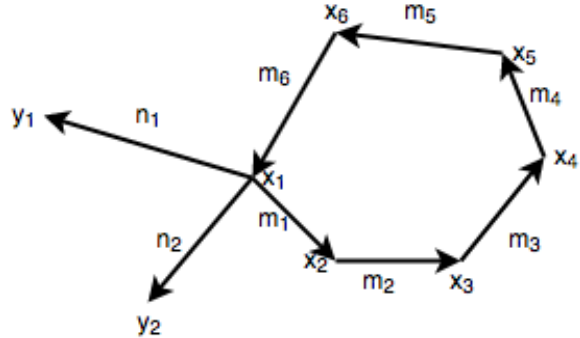}
    \caption{The graph $\widehat{\Gamma}$.  }
    \label{fig:a}
 \end{figure}   

We fix the signature $\nu_{1} = 0$ for $n_{1}$ and the signature $\mu_{1} =0$ for $m_{1}$, then other signatures are implicitly determined by the constraint at each internal vertex because of the formal Fourier convolution, namely $n_1+n_2 -m_6 + m_1 =0$ for $x_1$ and $m_{k+1} - m_{k} = 0$ for each $x_k$ ($2\leq k \leq 5$).
\begin{align*}
& G_{\mathbb{Z},\nu,\widehat{\Gamma}, \partial \widehat{\Gamma}}(\{y_1,y_2\},1)  \\
= & \int_{(\mathbb{R}/\mathbb{Z})^{6}}
\sum_{\substack{n_{k},m_l \in \mathbb{Z}\smallsetminus\{0\}\\\mathrm{sgn}(n_k) = (-1)^{\nu_k}\\ \mathrm{sgn}(m_l) = (-1)^{\mu_l}\\ n_1 + n_2 +m_1 - m_6 =  0\\ m_l = m_{l+1}, 1 \leq l \leq 5}}\frac{e^{2\pi i (\sum^{5}_{l=1}m_l(x_{l+1}-x_{l})+m_6(x_1 - x_6)+n_1(y_1-x_1)+n_2(y_2-x_1))}}{\prod^{6}_{l=1}\vert m_l\vert \cdot\vert n_1 \vert \cdot |n_2| }dx_1\cdots dx_{6}.
\end{align*}

By formal Fourier convolution, 
\[G_{\mathbb{Z},\nu,\widehat{\Gamma}, \partial \widehat{\Gamma}}(\{y_1,y_2\},1) = \sum_{n_1,m_1 \in \mathbb{N}\smallsetminus\{0\}}\frac{e^{2\pi i(n_1(y_1-y_2))}}{m_1^{6}n_1^{2}} = \zeta(6)\sum^{\infty}_{n=0}\frac{e^{2\pi i n(y_1-y_2)}}{n^2}, \]
and 
\[Z_{\mathbb{Z},\nu}(\widehat{\Gamma}, \partial \widehat{\Gamma}) = \sum_{n_1,m_1 \in \mathbb{N}\smallsetminus\{0\}}\frac{1}{m_1^{6}n_1^{2}} = \zeta(6)\zeta(2).\]
%By the shuffle relation, we get 
%\[Z_{\mathbb{Z},\nu}(\widehat{\Gamma}, \partial \widehat{\Gamma})  = \sum_{r+s= 8}(C^{5}_{r-1}+ C^{1}_{r-1})\zeta(s,r).\]
The result is equal to the value of 
$$Z_{\mathbb{Z},\nu}(\Gamma_b, \{s_1, s_2\}) = \sum_{\substack{n_1,m_1\in \mathbb{N}\smallsetminus\{0\}\\ n_2 = -n_1 }}\frac{1}{m^6_1 n^{2}_1} = \sum_{n_1,m_1 \in \mathbb{N}\smallsetminus\{0\}}\frac{1}{m_1^{6}n_1^{2}} ,$$ 
given by Figure \ref{fig:b}.

\begin{figure}[h]
 \centering
    \includegraphics[width=0.6\textwidth]{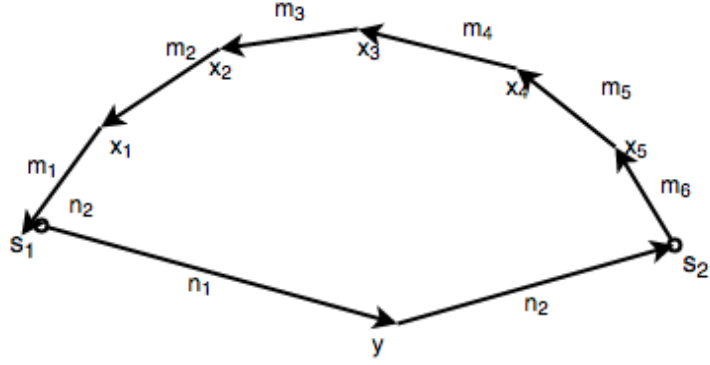}
    \caption{The graph $\Gamma_b$.  }
    \label{fig:b}
 \end{figure}   
 
From the discussion above, we can see that different graphs can deliver the same value.
\end{ex}

\begin{ex}
Let $\Gamma $ be the graph in Figure \ref{fig:00}. 
\begin{figure}[h]
 \centering
    \includegraphics[width=0.6\textwidth]{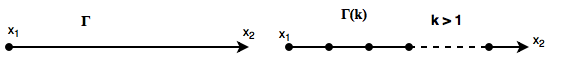}
    \caption{A plane tree of rank $1$.}
    \label{fig:00}
\end{figure}

The graph $\Gamma(\underline{k}) $ is just a chain obtained after adding $k-1$ points to $\Gamma $. The given sign is
$\nu_i = \nu$, $1 \leq i \leq k$.  
 Then
\[G_{I,\nu,\Gamma (k), \partial \Gamma(\underline{k})}(\{x_{v}\}_{v\in \partial \Gamma(\underline{k})}, 1) = \sum_{\substack{n\in \mathbb{Z}\smallsetminus\{0\} \\ \mathrm{sgn}(n) = (-1)^{\nu}}}\frac{e^{2\pi in(x_2-x_1)}}{|n|^{k}},\]
and 
\[Z_{I,\nu}(\Gamma (k), \partial \Gamma(\underline{k})) = \sum_{\substack{n\in \mathbb{Z}\smallsetminus\{0\} \\ \mathrm{sgn}(n) = (-1)^{\nu}}}\frac{1}{|n|^{k}} = \zeta(k).\]
Moreover, we even do not need sign $\nu$, then one obtain
\[Z_{I}(\Gamma (k), \partial \Gamma(\underline{k})) = \sum_{\substack{n\in \mathbb{Z}\smallsetminus\{0\}}}\frac{1}{|n|^{k}} = 2\zeta(k).\]
\end{ex}

\paragraph{Case of trees of internal valency $\geq 3$.}

\begin{ex}\label{ex:1}
Let $\Gamma_{1}$ be the graph in Figure \ref{fig:1}. 
\begin{figure}[h]
 \centering
    \includegraphics[width=0.25\textwidth]{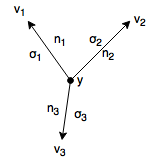}
    \caption{A plane trivalent tree with one internal vertex}
    \label{fig:1}
\end{figure}

The rank of $\Gamma_{1} $ is $ 2 = rank(H_{1}(\Gamma_{1}, \partial \Gamma_{1}))$. To each edge
$e_{i}\ (i = 1,2,3)$ we add $\sigma_{i} -1 \geq 0$ points. The only internal vertex is denoted by $y$, each external vertex $v_{i}$ is decorated by the variable $x_{v_{i}}$. For each $e_i \ (i = 1, 2)$, the given sign $\nu_i$ equals $0$. For the edge $e_3$, the sign $\nu_3 =1$. 
In fact, the constraint $n_1+n_2+n_3 = 0$ and $\nu_1 = \nu_2 = 0$ imply that $\nu_3 = 1$. 

$$G_{\mathbb{Z},\nu,\Gamma_{1}, \partial \Gamma_{1}}(\{x_{v}\}_{v\in \partial \Gamma_{1}}, 1) = \sum_{\substack{n_{1}+n_{2}+n_{3}=0,n_{i}\in \mathbb{Z}\smallsetminus \{0\}\\\mathrm{sgn}(n_{i})=(-1)^{\nu_i}}}\frac{e^{2\pi i( n_{1}x_{v_{1}}+n_{2}x_{v_{2}}+n_{3}x_{v_{3}})}}{|n_{1}|^{\sigma_{1}}|n_{2}|^{\sigma_{2}}|n_{3}|^{\sigma_{3}}},
$$
where $\nu_{1} = \nu_{2} = 0$, $\nu_3 = 1$. Then 
\begin{align*}
 & Z_{\mathbb{Z},\nu}(\Gamma_{1}, \partial \Gamma_{1}) \\
=  & \sum_{\substack{n_{1}+n_{2}+n_{3}=0,n_{i}\in \mathbb{Z}\smallsetminus \{0\}\\ \mathrm{sgn}(n_{i})=(-1)^{\mu_i}}}\frac{1}{|n_{1}|^{\sigma_{1}}|n_{2}|^{\sigma_{2}}|n_{3}|^{\sigma_{3}}} \\
= & \sum_{n_{1},n_{2}\in \mathbb{N}\smallsetminus \{0\}}\frac{1}{|n_{1}|^{\sigma_{1}}|n_{2}|^{\sigma_{2}}|n_{1}+n_{2}|^{\sigma_{3}}}
\end{align*}

\[Z_{\mathbb{Z},\nu}(\Gamma_{1}, \partial \Gamma_{1}) = \sum_{n_{1},n_{2}\in \mathbb{N}\smallsetminus\{0\}}\frac{1}{n_{1}^{\sigma_{1}}n_{2}^{\sigma_{2}}(n_{1}+n_{2})^{\sigma_{3}}} .\]
Recall the following Eisenstein's trick
\begin{formu}\label{for:1}
\[\frac{1}{m^{i}n^{j}} = \sum_{r+s = i+j}\bigg( \frac{C^{i-1}_{r-1}}{(m+n)^{r}n^{s}} + \frac{C^{j-1}_{r-1}}{(m+n)^{r}m^{s}} \bigg ). \]
then $\frac{1}{n_{1}^{\sigma_1}n_{2}^{\sigma_2}} = \sum_{r+s = \sigma_{1}+\sigma_{2}}\bigg( \frac{C^{\sigma_{1}-1}_{r-1}}{(n_{1}+n_{2})^{r}n_{2}^{s}} + \frac{C^{\sigma_{2}-1}_{r-1}}{(n_{1}+n_{2})^{r}n_{1}^{s}} \bigg ),\  r, s \geq 1.$
\end{formu}
%where $C^{a}_{b} = $ ${b}\choose {a}$.
Hence
\[ Z_{\mathbb{Z},\nu}(\Gamma_{1}, \partial \Gamma_{1}) = \sum_{n_{1},n_{2}\in \mathbb{N}\smallsetminus\{0\}}\sum_{r+s = \sigma_{1}+\sigma_{2}}\bigg( \frac{C^{\sigma_{1}-1}_{r-1}}{(n_{1}+n_{2})^{r + \sigma_{3}}n_{2}^{s}} + \frac{C^{\sigma_{2}-1}_{r-1}}{(n_{1}+n_{2})^{r+\sigma_{3}}n_{1}^{s}} \bigg ),\]
\[ = \sum_{r+s = \sigma_{1}+\sigma_{2}}\Bigg\{C^{\sigma_{1}-1}_{r-1}\bigg(\sum_{n_{1},n_{2}\in \mathbb{N}\smallsetminus\{0\}} \frac{1}{(n_{1}+n_{2})^{r + \sigma_{3}}n_{2}^{s}}\bigg) + C^{\sigma_{2}-1}_{r-1}\bigg(\sum_{n_{1},n_{2}\in \mathbb{N}\smallsetminus\{0\}}\frac{1}{(n_{1}+n_{2})^{r+\sigma_{3}}n_{1}^{s}} \bigg )\Bigg\}.\]
Since $$\sum_{n_{1},n_{2}\in \mathbb{N}\smallsetminus\{0\}} \frac{1}{(n_{1}+n_{2})^{r + \sigma_{3}}n_{2}^{s}} = \sum_{0<n_{2}< n'_{1}}\frac{1}{(n'_{1})^{r + \sigma_{3}}n_{2}^{s}} = \zeta(s,r+\sigma_{3}),$$
where $n'_{1} = n_{1}+n_{2}$. We can express the generalized multiple zeta value for $\Gamma_{1}$ as follows: 
$$Z_{\mathbb{Z},\nu}(\Gamma_{1}, \partial \Gamma_{1}) =   \sum_{r+s = \sigma_{1}+\sigma_{2}}\left(C^{\sigma_{1}-1}_{r-1}\zeta(s,r+\sigma_{3}) + C^{\sigma_{2}-1}_{r-1}\zeta(s,r+\sigma_{3}) \right). $$
$$=   \sum_{r+s = \sigma_{1}+\sigma_{2}}\bigg ( C^{\sigma_{1}-1}_{r-1} + C^{\sigma_{2}-1}_{r-1}\bigg)\zeta(s,r+\sigma_{3}) .  $$
\end{ex}
Here $\zeta(s, r+\sigma_{3})$ is a classical double zeta value of weight $r + s + \sigma_{3} = \sigma_{1} + \sigma_{2}+\sigma_{3}$, which means that $Z_{\mathbb{Z},\nu}(\Gamma_{1}, \partial \Gamma_{1})$ can be expressed as a $\mathbb{Z}$-linear combination of double zeta values of this weight.

\begin{ex} \label{ex:1'}
Let $\Gamma'_{1}$ be the diagram as in Figure \ref{fig:1'}. 
\begin{figure}[h]
 \centering
    \includegraphics[width=0.35\textwidth]{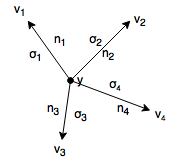}
    \caption{A plane trivalent tree with one internal vertex and 4 edges}
    \label{fig:1'}
\end{figure}    
The rank of $\Gamma'_1$ is $ rank(H_{1}(\Gamma'_{1}, \partial \Gamma'_{1})) = 3$. In fact, the $\Gamma'_1$ is no longer a plane trivalent tree. For each edge
$e_{i} ( 1\leq i \leq 4)$, we add $\sigma_{i} -1 \ (\sigma_{i}\geq 1)$ points. The only internal vertex is denoted by $y$, each external vertex $v_{i}$ is decorated by $x_{v_{i}}$. For each $e_i \ (i = 1, 2,3)$, the given sign $\nu_i$ equals $0$; for $e_4$ the sign $\nu_4 =1$.

$$G_{\mathbb{Z},\nu,\Gamma'_{1}, \partial \Gamma'_{1}}(\{x_{v}\}_{v\in \partial \Gamma'_1}, 1) = \sum_{\substack{n_{1}+n_{2}+n_{3}+ n_4=0,n_{i}\in \mathbb{Z}\smallsetminus \{0\};\\ \mathrm{sgn}(n_{i})=(-1)^{\nu_i}, 1 \leq i \leq 4}}\frac{e^{2\pi i( n_{1}x_{v_{1}}+n_{2}x_{v_{2}}+n_{3}x_{v_{3}}+n_4x_{v_{4}})}}{|n_{1}|^{\sigma_{1}}|n_{2}|^{\sigma_{2}}|n_{3}|^{\sigma_{3}}|n_{4}|^{\sigma_{4}}},
$$
where $\nu_{1} = \nu_{2} = \nu_{3} = 0$, $\nu_4=1$. Then 
$$ Z_{\mathbb{Z},\nu}(\Gamma'_{1}, \partial \Gamma'_{1}) = G_{I,\nu,\Gamma'_{1}, \partial \Gamma'_{1}}(\{0\}_{v\in \partial \Gamma'_1}, 1) $$

$$=  \sum_{\substack{n_{1}+n_{2}+n_{3}+n_4=0,n_{i}\in \mathbb{Z}\\ \mathrm{sgn}(n_j)=(-1)^{v_j}, 1 \leq j \leq 4}}\frac{1}{|n_{1}|^{\sigma_{1}}|n_{2}|^{\sigma_{2}}|n_{3}|^{\sigma_{3}}|n_{4}|^{\sigma_4}} $$
\[Z_{\mathbb{Z},\nu}(\Gamma'_{1}, \partial \Gamma'_{1}) = \sum_{n_{1},n_{2},n_3\in \mathbb{N}\smallsetminus\{0\}}\frac{1}{n_{1}^{\sigma_{1}}n_{2}^{\sigma_{2}}n_3^{\sigma_{3}}(n_{1}+n_{2}+n_{3})^{\sigma_4}} .\]

Firstly, we use Eisenstein's trick for 
\[ \frac{1}{n_2^{\sigma_2}n_3^{\sigma_3}} = \sum_{r_1 + s_1= \sigma_2 + \sigma_3} \left ( \frac{C^{\sigma_{2}-1}_{r_{1}-1}}{(n_2+n_3)^{r_{1}}n_{3}^{s_{1}}} + \frac{C^{\sigma_{3}-1}_{r_{1}-1}}{(n_2+n_3)^{r_1}n_{2}^{s_1}}\right) \]
Then 
$$Z_{\mathbb{Z},\nu}(\Gamma'_{1}, \partial \Gamma'_{1}) = $$
$$ \sum_{r_1 + s_1= \sigma_2 + \sigma_3}\sum_{n_{1},n_{2},n_3\in \mathbb{N}\smallsetminus\{0\}}  \left ( \frac{C^{\sigma_{2}-1}_{r_{1}-1}}{(n_{1}+n_{2}+n_{3})^{\sigma_4}(n_2+n_3)^{r_{1}}n_{3}^{s_{1}}n_{1}^{\sigma_1}} + \frac{C^{\sigma_{3}-1}_{r_{1}-1}}{(n_{1}+n_{2}+n_{3})^{\sigma_4}(n_2+n_3)^{r_1}n_{2}^{s_1}n_{1}^{\sigma_1}}\right) $$
Secondly, we use twice Eisenstein's trick for the terms involving $n_1$ and $(n_2 + n_3)$. Then we obtain
\[ Z_{\mathbb{Z},\nu}(\Gamma'_{1}, \partial \Gamma'_{1}) = \sum_{\substack{r_1 + s_1= \sigma_2 + \sigma_3\\r_2 + s_2 = r_1 + \sigma_1}} \sum_{n_{1},n_{2},n_3\in \mathbb{N}\smallsetminus\{0\}}\frac{C^{r_1-1}_{r_2-1}C^{\sigma_{2}-1}_{r_{1}-1}}{(n_{1}+n_{2}+n_{3})^{\sigma_4 + r_2}n_{3}^{s_{1}}n_{1}^{s_2}} \]

\[+ \sum_{\substack{r_1 + s_1= \sigma_2 + \sigma_3\\r_2 + s_2 = r_1 + \sigma_1}} \sum_{n_{1},n_{2},n_3\in \mathbb{N}\smallsetminus\{0\}} \frac{C^{\sigma_1-1}_{r_2-1}C^{\sigma_{2}-1}_{r_{1}-1}}{(n_{1}+n_{2}+n_{3})^{\sigma_4+ r_2}(n_2+n_3)^{s_2}n_{3}^{s_1}}\]

\[+ \sum_{\substack{r_1 + s_1= \sigma_2 + \sigma_3\\r'_2 + s'_2 = r_1 + \sigma_1}} \sum_{n_{1},n_{2},n_3\in \mathbb{N}\smallsetminus\{0\}}\frac{C^{r_1-1}_{r'_2-1}C^{\sigma_{3}-1}_{r_{1}-1}}{(n_{1}+n_{2}+n_{3})^{\sigma_4 + r'_{2}}n_{2}^{s_{1}}n_{1}^{s'_2}} \]

\[+ \sum_{\substack{r_1 + s_1= \sigma_2 + \sigma_3\\r'_2 + s'_2 = r_1 + \sigma_1}} \sum_{n_{1},n_{2},n_3\in \mathbb{N}\smallsetminus\{0\}} \frac{C^{\sigma_1-1}_{r'_2-1}C^{\sigma_{3}-1}_{r_{1}-1}}{(n_{1}+n_{2}+n_{3})^{\sigma_4 + r'_2}(n_2+n_3)^{s'_2}n_{2}^{s_1}}\]

Finally, we use Eisenstein's trick for the terms involving $n_1$ and $n_2$(respectively, $n_1$ and $n_3$). 
\[ Z_{\mathbb{Z}, \nu}(\Gamma'_{1}, \partial \Gamma'_{1}) = \sum_{\substack{r_1 + s_1= \sigma_2 + \sigma_3\\r_2 + s_2 = r_1 + \sigma_1\\r_3+s_3 = s_1 + s_2}} \sum_{n_{1},n_{2},n_3\in \mathbb{N}\smallsetminus\{0\}}\frac{C^{s_1-1}_{r_3-1}C^{r_1-1}_{r_2-1}C^{\sigma_{2}-1}_{r_{1}-1}}{(n_{1}+n_{2}+n_{3})^{\sigma_4+r_{2}}(n_1 + n_3)^{r_3}n_{1}^{s_3}} \]

\[+ \sum_{\substack{r_1 + s_1= \sigma_2 + \sigma_3\\r_2 + s_2 = r_1 + \sigma_1\\r_3+s_3 = s_1 + s_2}} \sum_{n_{1},n_{2},n_3\in \mathbb{N}\smallsetminus\{0\}}\frac{C^{s_2-1}_{r_3-1}C^{r_1-1}_{r_2-1}C^{\sigma_{2}-1}_{r_{1}-1}}{(n_{1}+n_{2}+n_{3})^{\sigma_4+r_{2}}(n_1 + n_3)^{r_3}n_{3}^{s_{3}}}\]

\[+ \sum_{\substack{r_1 + s_1= \sigma_2 + \sigma_3\\r_2 + s_2 = r_1 + \sigma_1}} \sum_{n_{1},n_{2},n_3\in \mathbb{N}\smallsetminus\{0\}} \frac{C^{\sigma_1-1}_{r_2-1}C^{\sigma_{2}-1}_{r_{1}-1}}{(n_{1}+n_{2}+n_{3})^{\sigma_4+r_2}(n_2+n_3)^{s_2}n_{3}^{s_1}}\]

\[+ \sum_{\substack{r_1 + s_1= \sigma_2 + \sigma_3\\r'_2 + s'_2 = r_1 + \sigma_1\\r'_3+s'_3 = s_1 +s'_2}} \sum_{n_{1},n_{2},n_3\in \mathbb{N}\smallsetminus\{0\}}\frac{C^{s_1-1}_{r'_3-1}C^{r_1-1}_{r'_2-1}C^{\sigma_{3}-1}_{r_{1}-1}}{(n_{1}+n_{2}+n_{3})^{\sigma_4 + r'_{2}}(n_1+n_{2})^{r'_{3}}n_{1}^{s'_3}} \]

\[+ \sum_{\substack{r_1 + s_1= \sigma_2 + \sigma_3\\r'_2 + s'_2 = r_1 + \sigma_1\\r'_3+s'_3 = s_1 +s'_2}} \sum_{n_{1},n_{2},n_3\in \mathbb{N}\smallsetminus\{0\}}\frac{C^{s'_2-1}_{r'_3-1}C^{r_1-1}_{r'_2-1}C^{\sigma_{3}-1}_{r_{1}-1}}{(n_{1}+n_{2}+n_{3})^{\sigma_4 + r'_{2}}(n_{1}+n_2)^{r'_3}n_{2}^{s'_{3}}} \]

\[+ \sum_{\substack{r_1 + s_1= \sigma_2 + \sigma_3\\r'_2 + s'_2 = r_1 + \sigma_1}} \sum_{n_{1},n_{2},n_3\in \mathbb{N}\smallsetminus\{0\}} \frac{C^{\sigma_1-1}_{r'_2-1}C^{\sigma_{3}-1}_{r_{1}-1}}{(n_{1}+n_{2}+n_{3})^{\sigma_4 + r'_2}(n_2+n_3)^{s'_2}n_{2}^{s_1}}\]
Then
\[Z_{\mathbb{Z},\nu}(\Gamma'_{1}, \partial \Gamma'_{1}) = \sum_{\substack{r_1 + s_1= \sigma_2 + \sigma_3\\r_2 + s_2 = r_1 + \sigma_1\\r_3+s_3 = s_1 + s_2}}\left(C^{s_1-1}_{r_3-1}C^{r_1-1}_{r_2-1}C^{\sigma_{2}-1}_{r_{1}-1} +C^{s_2-1}_{r_3-1}C^{r_1-1}_{r_2-1}C^{\sigma_{2}-1}_{r_{1}-1}\right)\zeta (s_3, r_3, \sigma_4 + r_2)  \]

\[+\sum_{\substack{r_1 + s_1= \sigma_2 + \sigma_3\\r'_2 + s'_2 = r_1 + \sigma_1\\r'_3+s'_3 = s_1 +s'_2}} \left(C^{s_1-1}_{r'_3-1}C^{r_1-1}_{r'_2-1}C^{\sigma_{3}-1}_{r_{1}-1} + C^{s'_2-1}_{r'_3-1}C^{r_1-1}_{r'_2-1}C^{\sigma_{3}-1}_{r_{1}-1}\right)\zeta(s'_3, r'_3, \sigma_4 + r'_2)\]

\[+ \sum_{\substack{r_1 + s_1= \sigma_2 + \sigma_3\\r_2 + s_2 = r_1 + \sigma_1}}C^{\sigma_1-1}_{r_2-1}C^{\sigma_{2}-1}_{r_{1}-1}\zeta(s_1,s_2, \sigma_4+r_2) + \sum_{\substack{r_1 + s_1= \sigma_2 + \sigma_3\\r'_2 + s'_2 = r_1 + \sigma_1}}C^{\sigma_1-1}_{r'_2-1}C^{\sigma_{3}-1}_{r_{1}-1}\zeta(s_1, s'_2, \sigma_4+r'_2) ,\]
where $\sigma_4 + r_2 + r_3 + s_3 = \sigma_4+r_2 + s_2 + s_1 =\sigma_1 + \sigma_2 + \sigma_3 +\sigma_4 $.\\
We can see that $Z_{\mathbb{Z},\nu}(\Gamma'_{1}, \partial \Gamma'_{1})$ is a $\mathbb{Z}$-linear combination of triple-zeta values of weight $\sigma_1 + \sigma_2 + \sigma_3 +\sigma_4$.
\end{ex}

\begin{ex} \label{ex:2}
Let $\Gamma_{2}$ be the diagram as in Figure \ref{fig:2}. 
\begin{figure}[h]
 \centering
    \includegraphics[width=0.55\textwidth]{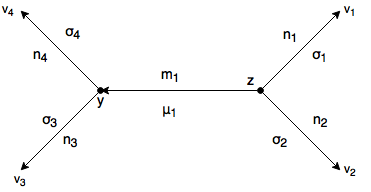}
    \caption{A plane trivalent tree with two internal vertices}
    \label{fig:2}
\end{figure}

 The rank is equal to $ rank(H_{1}(\Gamma_{2}, \partial \Gamma_{2})) = 3$. For each external edge $e_i\ (1 \leq i \leq 3)$, the given sign $\nu_i$ equals $0$, therefore $\nu_4 = 1$.

$$G_{I,\nu,\Gamma_{2}, \partial \Gamma_{2}}(\{x_{v}\}_{v\in S}, 1) = $$
$$\int_{(\mathbb{R}/\mathbb{Z})^{2}}\sum_{\substack{n_{i},m_i\in \mathbb{Z}\smallsetminus\{0\}\\\mathrm{sgn}(n_{i} )= (-1)^{\nu_i}}}\frac{e^{2\pi i((x_{v_{3}}-y)n_{3}+(x_{v_{4}}-y)n_{4})}}{|n_{3}|^{\sigma_{3}}|n_{4}|^{\sigma_{4}}}\frac{e^{2\pi i((x_{v_{1}}-z)n_{1}+(x_{v_{2}}-z)n_{2}+(y-z)m_{1})}}{|n_{1}|^{\sigma_{1}}|n_{2}|^{\sigma_{2}}|m_{1}|^{\mu_{1}}}dxdy $$
By the formal Fourier convolution, we get
$$G_{I,\nu,\Gamma_{2}, \partial \Gamma_{2}}(\{x_{v}\}_{v\in S}, 1) = \sum_{\substack{n_{1}+n_{2}+m_{1}=0, n_{3}+n_{4}-m_{1}=0; n_{i},m_{i}\in \mathbb{Z}\smallsetminus\{0\}\\\mathrm{sgn}(n_{i} )= (-1)^{\nu_i}}}\frac{e^{2\pi i( n_{1}x_{v_{1}}+n_{2}x_{v_{2}}+n_{3}x_{v_{3}} +n_{4}x_{v_{4}})}}{|n_{1}|^{\sigma_{1}}|n_{2}|^{\sigma_{2}}|n_{3}|^{\sigma_{3}}|n_{4}|^{\sigma_{4}}|m_{1}|^{\mu_{1}}}.$$
 We can see that
for the first internal vertex $z$, we have a constraint condition $n_{1}+n_{2}+m_{1} = 0$ and for the internal vertex $y$, we have $n_{3}+n_{4}-m_{1} = 0$, $n_1 + n_2 +n_3 +n_4 = 0$. 
\[Z_{I,\nu}(\Gamma_2,\partial \Gamma_2) = G_{I,\nu,\Gamma_{2}, \partial \Gamma_{2}}(\{0\}_{v\in S}, 1) =\sum_{n_{1},n_{2},n_{3}\in \mathbb{N}\smallsetminus\{0\}}\frac{1}{n_{3}^{\sigma_{3}}(n_{1}+n_{2}+n_{3})^{\sigma_{4}}}\frac{1}{n_{1}^{\sigma_{1}}n_{2}^{\sigma_{2}}(n_{1}+n_{2})^{\mu_{1}}} \]
Applying Eisenstein's trick \ref{for:1} again, we get
\[Z_{I,\nu}(\Gamma_2,\partial \Gamma_2) = \sum_{n_{1},n_{2},n_{3}\in \mathbb{N}\smallsetminus\{0\}}\frac{1}{n_{3}^{\sigma_{3}}(n_{1}+n_{2}+n_{3})^{\sigma_{4}}} \sum_{s_{1} +t_{1} = \sigma_{1}+\sigma_{2}}\bigg( \frac{C^{\sigma_{1}-1}_{s_{1}-1}}{(n_{1}+n_{2})^{s_{1} + \mu _{1}}n_{2}^{t_{1}}} + \frac{C^{\sigma_{2}-1}_{s_{1}-1}}{(n_{1}+n_{2})^{s_{1}+\mu _{1}}n_{1}^{t_{1}}} \bigg ) \]

$$=  \sum_{n_{1},n_{2},n_{3}\in \mathbb{N}\smallsetminus\{0\}}\sum_{s_{1} + t_{1} = \sigma_{1}+\sigma_{2}}\bigg( \textcircled{1} +\textcircled{2} \bigg ),$$
where 
\[\textcircled{1} = \frac{C^{\sigma_{1}-1}_{s_{1}-1}}{n_{3}^{\sigma_{3}}(n_{1}+n_{2}+n_{3})^{\sigma_{4}}(n_{1}+n_{2})^{s_{1} + \mu _{1}}n_{2}^{t_{1}}},\]
\[\textcircled{2} = \frac{C^{\sigma_{2}-1}_{s_{1}-1}}{n_{3}^{\sigma_{3}}(n_{1}+n_{2}+n_{3})^{\sigma_{4}}(n_{1}+n_{2})^{s_{1}+\mu _{1}}n_{1}^{t_{1}}}.\]

Since $$ \frac{1}{n_{3}^{\sigma_3}(n_{1} + n_{2})^{s_{1}+\mu_{1}}} = \sum_{s_{2}+t_{2} = \sigma_{3}+s_{1}+\mu_{1}}\bigg( \frac{C^{\sigma_{3}-1}_{s_{2}-1}}{(n_{1}+n_{2}+n_{3})^{s_{2}}(n_{1}+n_{2})^{t_{2}}} + \frac{C^{s_{1}+\mu_{1}-1}_{s_{2}-1}}{(n_{1}+n_{2}+n_{3}
)^{s_{2}}n_{3}^{t_{2}}} \bigg ),$$
then
 \[\textcircled{1} = \sum_{s_{2}+t_{2} = \sigma_{3}+s_{1}+\mu_{1}}\bigg( \frac{C^{\sigma_{1}-1}_{s_{1}-1}C^{\sigma_{3}-1}_{s_{2}-1}}{(n_{1}+n_{2}+n_{3})^{s_{2}+\sigma_{4}}(n_{1}+n_{2})^{t_{2}}n_{2}^{t_{1}}} + \frac{C^{\sigma_{1}-1}_{s_{1}-1}C^{s_{1}+\mu_{1}-1}_{s_{2}-1}}{(n_{1}+n_{2}+n_{3})^{s_{2}+\sigma_{4}}n_{3}^{t_{2}}n_{2}^{t_{1}}} \bigg ),\]
 
 \[ \textcircled{2} = \sum_{s_{2}+t_{2} = \sigma_{3}+s_{1}+\mu_{1}}\bigg( \frac{C^{\sigma_{2}-1}_{s_{1}-1}C^{\sigma_{3}-1}_{s_{2}-1}}{(n_{1}+n_{2}+n_{3})^{s_{2}+\sigma_{4}}(n_{1}+n_{2})^{t_{2}}n_{1}^{t_{1}}} + \frac{C^{\sigma_{2}-1}_{s_{1}-1}C^{s_{1}+\mu_{1}-1}_{s_{2}-1}}{(n_{1}+n_{2}+n_{3})^{s_{2}+\sigma_{4}}n_{3}^{t_{2}}n_{1}^{t_{1}}} \bigg ).\]
 
Since  $$\frac{1}{n_{3}^{t_2}n_{2}^{t_1}} = \sum_{s_3+t_3 = t_{1}+t_{2}}\bigg( \frac{C^{t_{2}-1}_{s_3-1}}{(n_{3}+n_{2})^{s_3}n_{2}^{t_3}} + \frac{C^{t_{1}-1}_{s_3-1}}{(n_{3}+n_{2})^{s_3}n_{3}^{t_3}} \bigg ),$$

$$\frac{1}{n_{3}^{t_2}n_{1}^{t_1}} = \sum_{s'_3+t'_3 = t_{1}+t_{2}}\bigg( \frac{C^{t_{2}-1}_{s'_3-1}}{(n_{3}+n_{1})^{s'_3}n_{1}^{t'_3}} + \frac{C^{t_{1}-1}_{s'_3-1}}{(n_{3}+n_{1})^{s'_3}n_{3}^{t'_3}} \bigg ),$$
we can rewrite $\textcircled{1}$ and $\textcircled{2}$ as follows. 
$$
 \textcircled{1} =  \sum_{s_{2}+t_{2} = \sigma_{3}+s_{1}+\mu_{1}}\frac{C^{\sigma_{1}-1}_{s_{1}-1}C^{\sigma_{3}-1}_{s_{2}-1}}{(n_{1}+n_{2}+n_{3})^{s_{2}+\sigma_{4}}(n_{1}+n_{2})^{t_{2}}n_{2}^{t_{1}}}  + \sum_{\substack{s_{2}+t_{2} = \sigma_{3}+s_{1}+\mu_{1} \\ s_3+t_3 = t_{1}+t_{2}}}\frac{C^{\sigma_{1}-1}_{s_{1}-1}C^{s_{1}+\mu_{1}-1}_{s_{2}-1}C^{t_{2}-1}_{s_3-1}}{(n_{1}+n_{2}+n_{3})^{s_{2}+\sigma_{4}}(n_{3}+n_{2})^{s_3}n_{2}^{t_3}} $$
 $$
 + \sum_{\substack{s_{2}+t_{2} = \sigma_{3}+s_{1}+\mu_{1} \\ s_3+t_3 = t_{1}+t_{2}}}\frac{C^{\sigma_{1}-1}_{s_{1}-1}C^{s_{1}+\mu_{1}-1}_{s_{2}-1}C^{t_{1}-1}_{s_3-1}}{(n_{1}+n_{2}+n_{3})^{s_{2}+\sigma_{4}}(n_{3}+n_{2})^{s_3}n_{3}^{t_3}}
 $$
 
 $$
 \textcircled{2} =  \sum_{s_{2}+t_{2} = \sigma_{3}+s_{1}+\mu_{1}}\frac{C^{\sigma_{2}-1}_{s_{1}-1}C^{\sigma_{3}-1}_{s_{2}-1}}{(n_{1}+n_{2}+n_{3})^{s_{2}+\sigma_{4}}(n_{1}+n_{2})^{t_{2}}n_{1}^{t_{1}}}  + \sum_{\substack{s_{2}+t_{2} = \sigma_{3}+s_{1}+\mu_{1} \\ s'_3+t'_3 = t_{1}+t_{2}}}\frac{C^{\sigma_{2}-1}_{s_{1}-1}C^{s_{1}+\mu_{1}-1}_{s_{2}-1}C^{t_{2}-1}_{s'_3-1}}{(n_{1}+n_{2}+n_{3})^{s_{2}+\sigma_{4}}(n_{3}+n_{1})^{s'_3}n_{1}^{t'_3}} $$
 $$
 + \sum_{\substack{s_{2}+t_{2} = \sigma_{3}+s_{1}+\mu_{1} \\ s'_3+t'_3 = t_{1}+t_{2}}}\frac{C^{\sigma_{2}-1}_{s_{1}-1}C^{s_{1}+\mu_{1}-1}_{s_{2}-1}C^{t_{1}-1}_{s'_3-1}}{(n_{1}+n_{2}+n_{3})^{s_{2}+\sigma_{4}}(n_{3}+n_{1})^{s'_3}n_{3}^{t'_3}}
 $$
Finally we obtain
\begin{formu}\label{for:2}
\[Z_{I,\nu}(\Gamma_2,\partial \Gamma_2) = \sum_{\substack{s_1 + t_1 = \sigma_{1}+\sigma_{2} \\ s_{2}+t_{2} = \sigma_{3}+s_{1}+\mu_{1}}}\sum_{n_1,n_2,n_3\in \mathbb{N}\smallsetminus\{0\}}\frac{C^{\sigma_{1}-1}_{s_{1}-1}C^{\sigma_{3}-1}_{s_{2}-1}}{(n_{1}+n_{2}+n_{3})^{s_{2}+\sigma_{4}}(n_{1}+n_{2})^{t_{2}}n_{2}^{t_{1}}} \]

\[+ \sum_{\substack{s_1+t_1= \sigma_{1} +\sigma_2\\ s_{2}+t_{2} = \sigma_{3}+s_{1}+\mu_{1} \\ s_3+t_3 = t_{1}+t_{2}}}\sum_{n_1,n_2,n_3\in \mathbb{N}\smallsetminus\{0\}}\left(\frac{C^{\sigma_{1}-1}_{s_{1}-1}C^{s_{1}+\mu_{1}-1}_{s_{2}-1}C^{t_{2}-1}_{s_3-1}}{(n_{1}+n_{2}+n_{3})^{s_{2}+\sigma_{4}}(n_{3}+n_{2})^{s_3}n_{2}^{t_3}} + \frac{C^{\sigma_{1}-1}_{s_{1}-1}C^{s_{1}+\mu_{1}-1}_{s_{2}-1}C^{t_{1}-1}_{s_3-1}}{(n_{1}+n_{2}+n_{3})^{s_{2}+\sigma_{4}}(n_{3}+n_{2})^{s_3}n_{3}^{t_3}}\right)
\]

\[+ \sum_{\substack{s_1 + t_1= \sigma_1 + \sigma_2 \\s_{2}+t_{2} = \sigma_{3}+s_{1}+\mu_{1}}}\sum_{n_1,n_2,n_3\in \mathbb{N}\smallsetminus\{0\}}\frac{C^{\sigma_{2}-1}_{s_{1}-1}C^{\sigma_{3}-1}_{s_{2}-1}}{(n_{1}+n_{2}+n_{3})^{s_{2}+\sigma_{4}}(n_{1}+n_{2})^{t_{2}}n_{1}^{t_{1}}} \]

\[+ \sum_{\substack{s_1+t_1 = \sigma_1 + \sigma_2 \\s_{2}+t_{2} = \sigma_{3}+s_{1}+\mu_{1} \\ s'_3+t'_3 = t_{1}+t_{2}}}\sum_{n_1,n_2,n_3\in \mathbb{N}\smallsetminus\{0\}}\left(\frac{C^{\sigma_{2}-1}_{s_{1}-1}C^{s_{1}+\mu_{1}-1}_{s_{2}-1}C^{t_{2}-1}_{s'_3-1}}{(n_{1}+n_{2}+n_{3})^{s_{2}+\sigma_{4}}(n_{3}+n_{1})^{s'_3}n_{1}^{t'_3}} + \frac{C^{\sigma_{2}-1}_{s_{1}-1}C^{s_{1}+\mu_{1}-1}_{s_{2}-1}C^{t_{1}-1}_{s'_3-1}}{(n_{1}+n_{2}+n_{3})^{s_{2}+\sigma_{4}}(n_{3}+n_{1})^{s'_3}n_{3}^{t'_3}}\right).\]
\end{formu}
However,
 $$ \sum_{n_1,n_2,n_3\in \mathbb{N}\smallsetminus\{0\}}\frac{1}{(n_{1}+n_{2}+n_{3})^{s_{2}+\sigma_{4}}(n_{1}+n_{2})^{t_{2}}n_{2}^{t_{1}}} 
= \sum_{0 < k_1 < k_2 <k_3} \frac{1}{k_3^{s_{2}+\sigma_{4}}k_{2}^{t_2}k_1^{t_1}} = \zeta(t_1,t_2, s_2 + \sigma_{4}),$$
where $k_{1} = n_1, k_2 = n_1 + n_2, k_3 = n_1 + n_2 +n_3$. 
Therefore
$$Z_{I,\nu}(\Gamma_2,\partial \Gamma_2) =  \sum_{\substack{s_1 + t_1 = \sigma_{1}+\sigma_{2} \\ s_{2}+t_{2} = \sigma_{3}+s_{1}+\mu_{1}}}\left(C^{\sigma_{1}-1}_{s_{1}-1}C^{\sigma_{3}-1}_{s_{2}-1} + C^{\sigma_{2}-1}_{s_{1}-1}C^{\sigma_{3}-1}_{s_{2}-1} \right)\zeta(t_1,t_2,s_2 + \sigma_4)  $$

$$ + \sum_{\substack{s_1+t_1= \sigma_{1} +\sigma_2\\ s_{2}+t_{2} = \sigma_{3}+s_{1}+\mu_{1} \\ s_3+t_3 = t_{1}+t_{2}}}\left(C^{\sigma_{1}-1}_{s_{1}-1}C^{s_{1}+\mu_{1}-1}_{s_{2}-1}C^{t_{2}-1}_{s_3-1} + C^{\sigma_{1}-1}_{s_{1}-1}C^{s_{1}+\mu_{1}-1}_{s_{2}-1}C^{t_{1}-1}_{s_3-1}\right)\zeta(t_3, s_3, s_{2}+\sigma_{4})$$

$$ + \sum_{\substack{s_1+t_1= \sigma_{1} +\sigma_2\\ s_{2}+t_{2} = \sigma_{3}+s_{1}+\mu_{1} \\ s'_3+t'_3 = t_{1}+t_{2}}}\left(C^{\sigma_{2}-1}_{s_{1}-1}C^{s_{1}+\mu_{1}-1}_{s_{2}-1}C^{t_{2}-1}_{s'_3-1} + C^{\sigma_{2}-1}_{s_{1}-1}C^{s_{1}+\mu_{1}-1}_{s_{2}-1}C^{t_{1}-1}_{s'_3-1}\right)\zeta(t'_3, s'_3, s_{2}+\sigma_{4}).$$

We have expressed our $Z_{I,\nu}(\Gamma_2,\partial \Gamma_2)$ as a $\mathbb{Z}$-linear combination of triple zeta  values $\zeta(t_1,t_2,s_2 + \sigma_4) $, whose weight is $t_1 + t_2 + s_2 +\sigma_4  = \sigma_1 + \sigma_2 + \sigma_3 + \mu_1$, and $\zeta(t_3, s_3, s_{2}+\sigma_{4}) $ and $\zeta(t'_3, s'_3, s_{2}+\sigma_{4})$, whose weights are also $\sigma_1 + \sigma_2 + \sigma_3 + \mu_1$.
\end{ex}

\subsection{Proof of theorems}

\paragraph{\textbf{Proof of Theorem} \ref{theo:1}}
Inspired by the previous examples, we will prove the theorem \ref{theo:1} by induction on the number of internal vertices of a given tree. For simplicity, we will first consider only plane trivalent trees. Later we will prove this theorem for any tree.  

\begin{proof}
(\textbf{I}). \textbf{The graph} $\Gamma$ \textbf{is a plane trivalent tree}.\\
Let $\Gamma$ be a given plane trivalent tree  with $N$ internal vertices $w_{j}\ (1 \leq j \leq N )$, then it has $N+2$ external vertices $v_{i}\ (1\leq i \leq N+2)$, $N-1$ internal edges and $N+ 2$ external edges. 

  The subdivision map $k$ is given by $k_{e_{i}} = \sigma_{e_{i}}\ (\sigma_{e_{i}} \geq 1)$ if $e_{i}\ (1\leq i \leq N+2) $ is an external edge with endpoint $v_i$, and $k_{f_j} = \mu_{f_j}\ (\mu_{f_j} \geq 1)$ if $f_j\ (1\leq j \leq N-1)$ is an internal edge. Moreover, for each edge, a sign $\nu_{e} \in \{0,1\}$ is given.

  If $\Gamma$ is a plane trivalent tree with $N$ internal vertices, then the rank $d= rank(H_1(\Gamma,\partial \Gamma))$ is equal to $N + 1$.

  In fact, we can see that the orientation of each edge has no importance by changing the sign for each edge. Moreover, we can also assume that for each external edge $e_i$ with the sign $\nu_{e_i} = 0, \ 1 \leq i \leq N+1$ (we shall see that this forces $\nu_{e_{N+2}} = 1$). It is easy to see that we will lose no generality.  
  
  By the definition of the generalized multiple zeta value, 
  \[ Z_{I,\nu}(\Gamma, \partial  \Gamma) =  \sum_{n_1, \ldots, n_{N+1}\in \mathbb{N}\smallsetminus\{0\}} \prod_{1\leq i \leq N+2}\frac{1}{|n_i|^{\sigma_i}}\prod_{1\leq j \leq N-1}\frac{1}{|m_j|^{\mu_{j}}}.\]
  For convenience, we will define a new quantity

\[ \mathscr{O}_{I,\nu}(\Gamma , \partial \Gamma , (n_i)_i, (m_j)_j) = \prod_{1\leq i \leq N+2}\frac{1}{|n_i|^{\sigma_i}}\prod_{1\leq j \leq N-1}\frac{1}{|m_j|^{\mu_{j}}} .\]
In fact for each internal vertex we have \[ \sum_{e\in E(\Gamma),v_{1}(e)=v}n_{e} - \sum_{e\in E(\Gamma),v_{0}(e)=v }n_{e} = 0,\]
  hence each  $m_j$ can be written as a linear combination of the $n_i$, thus
\[\mathscr{O}_{I,\nu}(\Gamma , \partial \Gamma , (n_i)_i, (m_j)_j) = \mathscr{O}_{I,\nu}(\Gamma , \partial \Gamma , (n_i)_i). \]
If there is no ambiguity, for simplicity we write:
\[ \mathscr{O}_{I,\nu}(\Gamma , \partial \Gamma ) = \mathscr{O}_{I,\nu}(\Gamma , \partial \Gamma , (n_i)_i).\]
Then
$$
Z_{I,\nu}(\Gamma, \partial  \Gamma) = \sum_{\substack{n_{i}\in \mathbb{N}\smallsetminus\{0\},\\1 \leq i \leq d}} \mathscr{O}_{I,\nu}(\Gamma, \partial \Gamma).
$$

\textbf{Step 1} : By the results of the above examples, we know that Therem \ref{theo:1} holds when $N = 1, 2$.

\paragraph{\textbf{Inductive hypothesis}}  If $N = n \ (n\geq 1$), the theorem is true. Moreover, we assume that: 

 The generalized multiple zeta value $Z_{I,\nu}(\Gamma, \partial  \Gamma)$ can be written as follows
\[Z_{I,\nu}(\Gamma, \partial  \Gamma) =  \sum_{n_1, \ldots, n_{N+1}\in \mathbb{N}\smallsetminus\{0\}}\sum_{\gamma \in S_d}
\sum_{\substack{t^{\gamma}_{i}\in \mathbb{N}\smallsetminus\{0\}\\1\leq i \leq d}}\frac{C_{\gamma,t^{\gamma}_i}}{n^{t^{\gamma}_1}_{\gamma \cdot 1}(n_{\gamma\cdot 1}+n_{\gamma\cdot 2})^{t^{\gamma}_2}\cdots (n_{\gamma \cdot 1}+\ldots + n_{\gamma \cdot d})^{t^{\gamma}_{d}}},\] 
which implies that
\[Z_{I,\nu}(\Gamma, \partial \Gamma ) = \sum_{\gamma \in S_d}\sum_{\substack{t^{\gamma}_{i}\\1\leq i \leq d}} C_{\gamma,t^{\gamma}_i} \zeta(t^{\gamma}_{1}, \cdots, t^{\gamma}_{d}), \]
where $S_d$ is the symmetric group, $C_{\gamma, t^{\gamma}_i} \in \mathbb{Z}$ is a constant depending on $\gamma$ and $t^{\gamma}_{i}$. The upper mute symbol $\gamma$ of $t^{\gamma}_{j}$ implies the dependence of $\gamma$, and 
\[ t^{\gamma}_1 + \ldots + t^{\gamma}_{d} = \sum_{e}\sigma_{e} + \sum_{f}\mu_{f}, \quad \forall \gamma \in S_d,\]
and 
\[ t^{\gamma}_{i} \geq 1, \quad \forall \gamma \in S_d, \quad 1 \leq i \leq d.\]

Therefore the first sum 
$$\sum_{\substack{t^{\gamma}_{i} \in \mathbb{N}\smallsetminus \{0\} \\ 1 \leq i \leq d}} $$
is a finite sum.

The fact that Theorem \ref{theo:1} holds means that we have a new expression for $\mathscr{O}_{I,\nu}(\Gamma ,\partial \Gamma)$, namely 
 
\[ \mathscr{O}_{I,\nu}(\Gamma, \partial \Gamma) = \sum_{\gamma \in S_d}\sum_{\substack{t^{\gamma}_{i} \\1 \leq i \leq d}} \frac{C_{\gamma,t^{\gamma}_{i}}}{n^{t^{\gamma}_1}_{\gamma \cdot 1}(n_{\gamma\cdot 1}+n_{\gamma\cdot 2})^{t^{\gamma}_2}\cdots (n_{\gamma \cdot 1}+\ldots + n_{\gamma \cdot d})^{t^{\gamma}_{d}}}, \]
\[ t^{\gamma}_1 + \ldots + t^{\gamma}_{d} = \sum_{e}\sigma_{e} + \sum_{f}\mu_{f}, \quad \forall \gamma \in S_d.\]

 \textbf{Step 2}: Now we will prove the case of $N = n+1$.
 
For a plane trivalent tree, if the number of internal vertices is increased by $1$, then the rank of the tree is increased by $1$, too. 
 
Now we give a clockwise order for all external vertices $v_i$. We will also give an order for all internal vertices, such that  the internal vertex $w_{N}$, decorated by the variable $x_{N}$, is connected with the two external vertices $v_{N+1}$ and $v_{N+2}$ by external edges $\overrightarrow{e_{N+1}} = (w_N\longrightarrow v_{N+1})$ and $\overrightarrow{e_{N+2}} = (w_N \longrightarrow v_{N+2})$. 

\begin{figure}[h]
 \centering
    \includegraphics[width=0.45\textwidth]{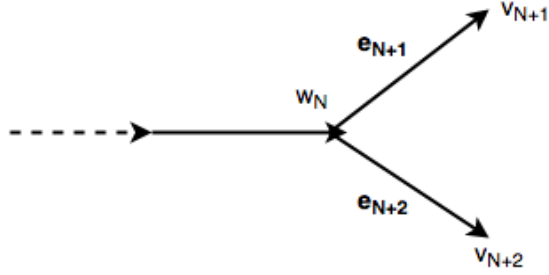}
    \caption{the internal vertex $w_{N}$ is connected with the two external vertices $v_{N+1}$ and $v_{N+2}$. }
    \label{fig:X}
 \end{figure}   
 
In order to deduce the case $N = n+1$ from the case $N -1 = n$, we do an operation: we cut down the internal edge $f_{N-1}$, one of whose ends is the internal vertex $w_{N}$ and associate a new external vertex denoted as $v'_{N+1}$ and denote the new external edge as $e'_{N+1}$ to which we associate $n_{e'_{N+1}} = m_{N-1}$ and the subdivision $\mu_{N-1}$, then we build a new plane trivalent tree $\Gamma '$ with $N-1 = n$ internal vertices and whose rank is $d-1$, where $d$ is the rank of $\Gamma $. 

\begin{figure}[h]
 \centering
    \includegraphics[width=1.0\textwidth]{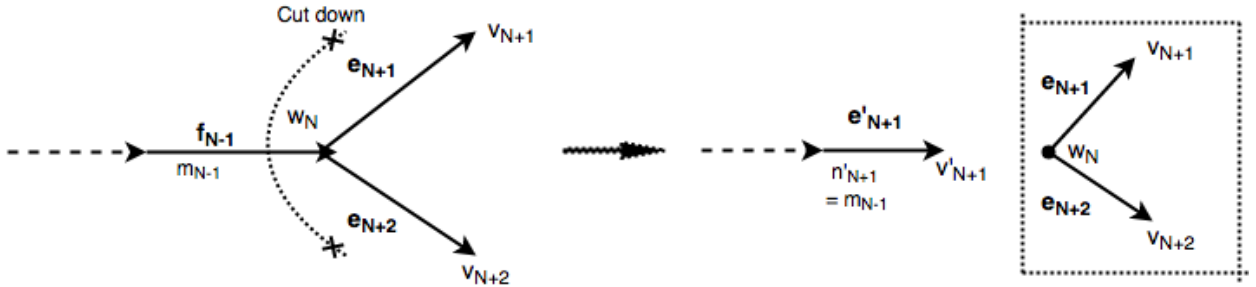}
    \caption{cutting down the internal edge $f_{N-1}$ to get a new tree $\Gamma'$. }
    \label{fig:X1}
 \end{figure}   

In the definition of 
\[ G_{I, \Gamma, S}(\{x_{v}\}_{v\in S}, 1) = \int_{(F_{\mathbb{R}}/I)^{V(\Gamma)\smallsetminus S}}\prod_{e\in E(\Gamma)}g_{I}(x_{v_{0}(e)}-x_{v_{1}(e)},1)\prod_{v\in V(\Gamma)\smallsetminus S}dx_{v}, \]
where
\[g_{I}(x_{v_{0}(e)}-x_{v_{1}(e)},1) = \lim_{\delta \rightarrow 0^{+}} \sum_{h_e\in I^{\ast}\backslash \{0\}}\frac{e^{2\pi i(h_e(x_{v_{0}(e)}-x_{v_{1}(e)}))}}{||uh_e||^{r+\delta}}.\]
%and $F =\mathbb{Q}\ (r =1)$, $I^{\ast} = \mathbb{Z}$.
We denote $h_e$ by $n_e$ if $e$ is external and denote $h_e$ by $m_e$ if $e$ is internal. And \[Z_{I,\nu}(\Gamma ,\partial \Gamma ) = G_{I,\nu,\Gamma,\partial \Gamma}(\{0\}_{v\in \partial \Gamma},1).\]

Then by formal Fourier convolution, we have
$$
G_{I,\nu,\Gamma,\partial \Gamma}(\{0\}_{v\in \partial \Gamma},1)= $$
$$\sum_{\substack{n_i \in \mathbb{Z}\smallsetminus\{0\}\\ \mathrm{sgn}(n_i) = (-1)^{\nu_{i}}, 1\leq i \leq d}}\frac{1}{|n_{N+1}|^{\sigma_{N+1}}|n_{N+2}|^{\sigma_{N+2}}|m_{N-1}|^{\mu_{N-1}}}\prod_{1\leq i \leq N}\frac{1}{|n_i|^{\sigma_i}}\prod_{1\leq j\leq N-2}\frac{1}{|m_j|^{\mu_{j}}},
$$
where $d = N+1$  and on each internal vertex, we have a constraint $\pi_v = 0$ as in (\ref{const:1}).

In fact, it is not difficult to see that 
\[ \frac{1}{|m_{N-1}|^{\mu_{N-1}}}\prod_{1\leq i \leq N}\frac{1}{|n_i|^{\sigma_i}}\prod_{1\leq j\leq N-2}\frac{1}{|m_j|^{\mu_{j}}} = \mathscr{O}_{I,\nu}(\Gamma ', \partial \Gamma '),\]
and
\[Z_{I,\nu}(\Gamma ', \partial \Gamma ') = \sum_{\substack{n_i \in \mathbb{N}\smallsetminus\{0\}, 1\leq i \leq N}}\mathscr{O}_{I,\nu}(\Gamma ', \partial \Gamma '). \]
Then 
$$
G_{I,\nu,\Gamma,\partial \Gamma}(\{0\}_{v\in \partial \Gamma},1)= \sum_{\substack{n_i \in \mathbb{N}\smallsetminus\{0\}, 1\leq i \leq N+1}}\frac{1}{|n_{N+1}|^{\sigma_{N+1}}|n_{N+2}|^{\sigma_{N+2}}}\cdot \mathscr{O}_{I,\nu}(\Gamma ', \partial \Gamma ').
$$

Since the number of the internal vertices of $\Gamma '$ is $n$, then the theorem for $\Gamma '$ holds by the inductive hypothesis.

Hence we have the following equality: 
\[ \mathscr{O}_{I,\nu}(\Gamma ', \partial \Gamma ') = \sum_{\gamma \in S_{d-1}}\sum_{t^{\gamma}_1 , \ldots ,t^{\gamma}_{d-1} } \frac{\tilde{C}_{\gamma, (t^{\gamma}_{i})_{i }}}{n^{t^{\gamma}_1}_{\gamma \cdot 1}(n_{\gamma\cdot 1}+n_{\gamma\cdot 2})^{t^{\gamma}_2}\cdots (n_{\gamma \cdot 1}+\ldots + n_{\gamma \cdot (d-1)})^{t^{\gamma}_{d-1}}}, \]
where $$t^{\gamma}_1 + \ldots + t^{\gamma}_{d-1} = \sum_{1\leq i \leq N}\sigma_{i} + \sum_{1\leq j\leq N - 1}\mu_{j}, \qquad \forall \gamma \in S_{d-1}, (d = N+1). $$
Now we need to calculate  
\[P_{\gamma} = \frac{1}{|n_{N+1}|^{\sigma_{N+1}}|n_{N+2}|^{\sigma_{N+2}}} \times \frac{\tilde{C}_{\gamma,(t^{\gamma}_{i})_i}}{n^{t^{\gamma}_1}_{\gamma \cdot 1}(n_{\gamma\cdot 1}+n_{\gamma\cdot 2})^{t^{\gamma}_2}\cdots (n_{\gamma \cdot 1}+\ldots + n_{\gamma \cdot (d-1)})^{t^{\gamma}_{d-1}}}, \]
from which we will deduce 
$$Z_{I,\nu}(\Gamma , \partial \Gamma ) =  \sum_{\substack{n_i \in \mathbb{N}\smallsetminus\{0\}, \\ 1\leq i \leq N+1}}\sum_{\gamma \in S_{d-1} }P_{\gamma}.$$
Since the rank $d$ of a plane trivalent tree with $N$ internal vertices equals $N+1$. Then 
$$Z_{I,\nu}(\Gamma , \partial \Gamma ) =  \sum_{\substack{n_i \in \mathbb{N}\smallsetminus\{0\}, \\1\leq i \leq N+1}}\sum_{\gamma \in S_{N} }P_{\gamma}.$$

\textbf{Step 3}: \textbf{Calculation of $P_{\gamma}$}. 

For simplicity, we assume that $\gamma = 1$. In fact, this assumption will be no loss of generality. 
\[P_1 =  \frac{1}{n_{N+1}^{\sigma_{N+1}}|n_{N+2}|^{\sigma_{N+2}}} \times \frac{\tilde{C}_{1,(t^1_{i})_i}}{n^{t^{1}_1}_{1}(n_{1}+n_{2})^{t^{1}_2}\cdots (n_{1}+\ldots + n_{N})^{t^{1}_{N}}} .\]
For any plane tree $\Gamma $, we have 
$$ n_1 + n_2 + \ldots + n_{N+1} + n_{N+2} = 0.  $$
Then 
\[ P_1 =  \frac{1}{n_{N+1}^{\sigma_{N+1}}(n_1+\ldots + n_{N+1})^{\sigma_{N+2}}} \times \frac{\tilde{C}_{1,(t^1_{i})_i}}{n^{t^{1}_1}_{1}(n_{1}+n_{2})^{t^{1}_2}\cdots (n_{1}+n_2 +\ldots + n_{N})^{t^{1}_{N}}} .\]
Now we will apply $N$ times the Eisenstein trick (\ref{for:1}). We will introduce several pieces of notation in order to simplify the demonstration.
\begin{note}\label{notation}
(1) We will write $Eis(a^{s_1}, b^{s_2})$ for the operation
$$\frac{1}{a^{s_1}b^{s_2}} = \sum_{r_1 + r_2 = s_1 + s_2}\frac{C^{s_1-1}_{r_1-1}}{(a+b)^{r_1}b^{r_2}} + \frac{C^{s_2-1}_{r_1-1}}{(a+b)^{r_1}a^{r_2}}. $$
(2) We define 
\[ n(k_1, k_2)= \sum^{k_2}_{j=k_1} n_j, \  k_1 < k_2 ;\]
\[ n_{k} = n(k, k).\]
\end{note}
Then we have 
\[  P_1 =  \frac{\tilde{C}_{1,(t^1_{i})_i}}{n_{N+1}^{\sigma_{N+1}}n(1, N+1)^{\sigma_{N+2}}\prod^{N}_{j=1}n(1,j)^{t^{1}_j}}.  \]

After applying $Eis\left (n^{\sigma_{N+1}}_{N+1}, n(1,N)^{t^1_N}\right )$, we obtain
%\[\frac{1}{n_{N+1}^{\sigma_{N+1}}n(1,N)^{t^{1}_{N}}} = \sum_{r_1+s_1 = \sigma_{N+1}+t^1_N} \frac{C^{\sigma_{N+1}-1}_{r_1-1}}{n(1,N+1)^{r_1}n(1,N)^{s_1}} 
% + \sum_{r_1+s_1 = \sigma_{N+1}+t^1_N} \frac{C^{t^1_N-1}_{r_1-1}}{n(1,N+1)^{r_1}n_{N+1}^{s_1}} .\]

\[P_1 = a^1_1 + b^1_1,\]
where 
\begin{align*}
a_1 & = \sum_{r_1+s_1 = \sigma_{N+1}+t^1_N} \frac{C^{\sigma_{N+1}-1}_{r_1-1}\tilde{C}_{1,(t^{1}_i)_i}}{n(1,N)^{s_1}n(1,N+1)^{r_1+ \sigma_{N+2}}\prod^{N-1}_{j=1}n(1,j)^{t^1_j}}; \\
b^1_1 & = \sum_{r_1+s_1 = \sigma_{N+1}+t^1_N} \frac{C^{t^1_{N}-1}_{r_1-1}\tilde{C}_{1,(t^1_i)_i}}{n_{N+1}^{s_1}n(1,N+1)^{r_1+ \sigma_{N+2}}\prod^{N-1}_{j=1}n(1,j)^{t^1_j}}.
\end{align*}
Here the lower index $1$ for $a^1_1$ refers to $P_1$ and the upper index $1$ indicates the 1st time use of Eisenstein's trick.

Denote $C_{1, (t^1_i)_i} = C^{\sigma_{N+1}-1}_{r_1-1}\tilde{C}_{1,(t^1_i)_i} $. 
Now we see that 
\begin{align*}
\sum_{\substack{n_i\in \mathbb{N}\smallsetminus\{0\}\\ 1 \leq i \leq N+1}} a^1_1 & = \sum_{r_1+s_1 = \sigma_{N+1}+t^1_N}\sum_{\substack{n_i\in \mathbb{N}\smallsetminus\{0\}\\1 \leq i \leq N+1}}\frac{C_{1,(t^1_i)_i}}{n(1,N)^{s_1}n(1,N+1)^{r_1+ \sigma_{N+2}}\prod^{N-1}_{j=1}n(1,j)^{t^1_j}} \\
&  = \sum_{r_1+s_1 = \sigma_{N+1}+t^1_N}C_{1,(t^1_i)_i}\zeta(t^1_1, \cdots, t^1_{N-1}, s_1, r_1+\sigma_{N+2}), 
\end{align*}
where $$t^1_1 + \cdots + t^1_{N-1} + s_1+ r_1+\sigma_{N+2} = \sum_{1\leq j \leq N}t^1_j + \sigma_{N+1} + \sigma_{N+2}.$$
By the inductive hypothesis
$$t^{\gamma}_1 + \ldots + t^{\gamma}_{d-1} = \sum_{1\leq i \leq N}\sigma_{i} + \sum_{1\leq j\leq N - 1}\mu_{j}, \quad \quad \forall \gamma \in S_{d-1}, \quad (d = N+1). $$
Then 
\[t^1_1 + \cdots + t^1_{N-1} + s_1+ r_1+\sigma_{N+2} = \sum_{1\leq i \leq N+2}\sigma_{i} + \sum_{1\leq j\leq N - 1}\mu_{j}.\]

$\zeta(t^1_1, \cdots, t^1_N, s_1, r_1+\sigma_{N+2})$ is a $(N+1)$-tuple zeta value of weight $\sum_{1\leq i \leq N+2}\sigma_i + \sum_{1\leq j \leq N-1}\mu_{j}$.
\\

Let us continue this procedure.

In fact, for any $2\leq k \leq N -1 $, after the $k$-th time use of Eisenstein's trick for $b^{k-1}_1$, 
\[Eis\left( n_{N+1}^{s_{(k-1)}}, n(1, N-(k-1))^{t^{1}_{N-(k-1)}}\right) \]
%\[= \sum_{\substack{r_k+s_k  = s_{k-1}+t^1_{N-(k-1)}}} \frac{C^{s_{k}-1}_{r_k-1}}{(n(1, N-(k-1))+ n_{N+1})^{r_k}n(1, N-(k-1))^{s_k}} \]
%\[ + \sum_{r_k+s_k = s_{k-1}+t^1_{N-(k-1)}} \frac{C^{t^1_{N-(k-1)}-1}_{r_k-1}}{(n(1, N-(k-1)) + n_{N+1})^{r_k}n_{N+1}^{s_k}}, \] 
we obtain two sums $a^{k}_1$ and $b^k_1$, 

\[ a^{k}_1 = \sum_{\substack{r_1+s_1=\sigma_{N+1}+t^{1}_{N}\\ r_2 + s_2 = s_1 + t^{1}_{N-1} \\ \cdot \\ \cdot \\ \cdot \\ r_k+s_k = s_{k-1}+t^{1}_{N-(k-1)}}}C^{s_{k-1}-1}_{r_k-1}C^{t^{1}_{N-(k-2)}}_{r_{k-1}-1}\cdots C^{t^{1}_{N-1}}_{r_1-1}\tilde{C}_{1,(t^1_i)_i} \times M_{(k)}, \]
where
$$ M_{(k)} = 
\frac{1}{\left[\prod^{N-k}_{j=1}n(1, j)^{t^1_j}\right]n(1,N-k+1)^{s_k}\left[\prod^{N-1}_{j=N-k+1}(n(1,j) +n_{N+1})^{r_{N-j+1}}\right]n(1, N+1)^{r_1+\sigma_{N+2}}} 
,$$
and
\[b^k_1 = \sum_{\substack{r_1+s_1 \\= \sigma_{N+1}+t^1_N;\\r_2+s_2 \\ =s_1+t^1_{N-1};\\ \cdots \\ r_k+s_k \\ = s_{k-1} + t^1_{N-k+1}}}
\frac{C^{t^1_{N-k+1}}_{r_k-1}\cdots C^{t^1_{N}-1}_{r_1-1}\tilde{C}_{1,(t^1_i)_i}}{n^{s_{k}}_{N+1}\prod^{N-k}_{j=1}n(1, j)^{t^1_j}
\prod^{N-1}_{j=N-k+1}(n(1, j) + n_{N+1})^{r_{N-j+1}}n(1, N+1)^{r_1+\sigma_{N+2}}}.\]

%Then we apply the $(k+1)$-th use of Eisenstein's trick for $b^k_1$
% $$Eis\left( n_{N+1}^{s_{k}}, n(1, N-k)^{t^{1}_{N-k}}\right) ,$$
%we obtain
%\[b^k_1 = a^{k+1}_1 + b^{k+1}_1.\] 
 
The final step is the $N$-th use of Eisenstein's trick $Eis\left( n_{N+1}^{s_{N}}, n_1^{t^{1}_{1}}\right) $ for 
$b^{N-1}_{1}$,
%\[= \sum_{\substack{r_1+s_1=\sigma_{N+1}+t^{1}_{N}\\ r_2 + s_2 = s_1 + t^{1}_{N-1} \\ \cdot \\ \cdot \\ \cdot \\ r_{N-1}+s_{N-1} = s_{N-2}+t^{1}_{2}}}\frac{C^{t^1_{2}}_{r_{N-1}-1}\cdots C^{t^1_{N}-1}_{r_1-1}\tilde{C}_{1,(t^1_i)_i}}{n^{t^1_1}_{1}n^{s_{N-1}}_{N+1}
%\prod^{N-1}_{j= 2}(n(1, j) + n_{N+1})^{r_{N-j+1}}n(1, N+1)^{r_1+\sigma_{N+2}}},\]

then we obtain 
$$b^{N-1}_1 = a^{N}_{1} + b^{N}_1,$$

\[a^{N}_1 = \sum_{\substack{r_1+s_1=\sigma_{N+1}+t^{1}_{N}\\ r_2 + s_2 = s_1 + t^{1}_{N-1} \\ \cdot \\ \cdot \\ \cdot \\ r_N+s_N = s_{N-1}+t^{1}_{1}}}C^{s_{N-1}-1}_{r_{N}-1}C^{t^{1}_{2}}_{r_{N-1}-1}\cdots C^{t^{1}_{N-1}}_{r_1-1}\tilde{C}_{1,(t^1_i)_i}\times M_{(N)}\]
where 
\[ M_{(N)} = \frac{1}{n^{s_N}_{1}\left[\prod^{N-1}_{j=1}(n(1, j) +n_{N+1})^{r_{N-j+1}}\right]n(1, N+1)^{r_1+\sigma_{N+2}}}\]
and

\[b^{N}_{1} = \sum_{\substack{r_1+s_1=\sigma_{N+1}+t^{1}_{N}\\ r_2 + s_2 = s_1 + t^{1}_{N-1} \\ \cdot \\ \cdot \\ \cdot \\ r_N+s_N = s_{N-1}+t^{1}_{1}}}C^{t^1_{1}-1}_{r_N-1}\cdots C^{t^1_{N}-1}_{r_1-1}\tilde{C}_{1,(t^1_i)_i}\times M_{(N+1)},\] 
where
\[M_{(N+1)} = \frac{1}{n^{s_{N}}_{N+1}
\prod^{N-1}_{j= 1}(n(1, j) + n_{N+1})^{r_{N-j+1}}n(1, N+1)^{r_1+\sigma_{N+2}}}.\]

In conclusion, after $N$ applications of Einstein's trick we obtain
\[P_1 = a^1_1 + b^1_1 = a^1_1 + a^2_1 + b^2_1= \cdots = \left(\sum_{1\leq j\leq k} a^{j}_{1} \right) + b^{k}_{1}= \cdots = \left(\sum_{1\leq l \leq N}a^{l}_1\right) + b^{N}_1.\]
We need the following lemma to deduce that $P_1$ is a finite linear combination of MZVs.

\begin{lem}
The sums $\sum_{n_i\in \mathbb{N}\smallsetminus\{0\}} a^{k}_1$ (for each $2 \leq k \leq N$) and $\sum_{n_i\in \mathbb{N}\smallsetminus\{0\}} b^{N}_1$ are finite $\mathbb{Z}$-linear combinations of $(N+1)$-tuple zeta values.
\end{lem}

\begin{proof}
Let us define $\tau _1 = Id$ and $\tau_{k}(2 \leq  k \leq N)$ is the permutation 
\[
  \tau_{k} = \bigl(\begin{smallmatrix}
     N-k+2 & N-k+3 & N-k+4 & \cdots & N & N+1\\    
    N+1 & N-k+2 & N-k+3 & \cdots &  N-1 & N
  \end{smallmatrix}\bigr)  \in S_{N+1},
\]
\[ \tau_{N+1} = \bigl(\begin{smallmatrix}
     1 & 2 & j & \cdots & N & N+1\\    
    2 & 3 & j+1 & \cdots &  N+1 & 1
  \end{smallmatrix}\bigr)  \in S_{N+1},\]
where $S_{N+1}$ is the symmetric group of $N+1$ elements and define 
\[\widetilde{M}_{(k)} = 
\frac{1}{\left[\prod^{N-k}_{j=1}n(1, j)^{t^1_j}\right]n(1,N-k+1)^{s_k}\left[\prod^{N}_{j=N-k+2}(n(1, j))^{r_{N-j+2}}\right]n(1, N+1)^{r_1+\sigma_{N+2}}},\]
\[ 2 \leq k \leq N-1;\]
\[\widetilde{M}_{(N)} = \widetilde{M}_{(N+1)} = 
\frac{1}{n_1^{s_N}\left[\prod^{N}_{j= 2}(n(1, j))^{r_{N-j+2}}\right]n(1, N+1)^{r_1+\sigma_{N+2}}}.\]
%and
%\[\widetilde{M}_{(N+1)} =\frac{1}{n^{s_{N+1}}_{1}(n_1+n_{2})^{r_N+r_{N+1}}
%\prod^{N}_{j= 3}(n(1, j) )^{r_{N-j+2}}n(1, N+1)^{r_1+\sigma_{N+2}}} .\]

By the definition of multiple zeta value, we know that
\[\sum_{\substack{n_j\in \mathbb{N}\smallsetminus\{0\}\\ 1 \leq j \leq N+1}} \widetilde{M}_{(k)} = \zeta (t^{1}_{1},\cdots, t^{1}_{N-k}, s_k, r_{k},\cdots, r_2, r_1+\sigma_{N+2}), \]
\[2 \leq k \leq N-1,\]
and 
\[\sum_{\substack{n_j\in \mathbb{N}\smallsetminus\{0\}\\ 1 \leq j \leq N+1}} \widetilde{M}_{(N)} = \zeta(s_N, r_N, r_{N-1}, \cdots, r_2, r_1+\sigma_{N+2}).\]
and
\[t^{1}_{1}+\cdots + t^{1}_{N-k}+ s_k + r_{k} +\cdots+ r_2+ r_1+\sigma_{N+2}  = \sum_{1 \leq i \leq N+2}\sigma_{i} + \sum_{1\leq j \leq N-1}\mu_{j}, \]
\[s_{N} + r_N + \cdots + r_2 + r_1 + \sigma_{N+2} = \sum_{1 \leq i \leq N+2}\sigma_{i} + \sum_{1\leq j \leq N-1}\mu_{j}.\] 

$\zeta(t^{1}_{1},\cdots, t^{1}_{N-k}, s_k, r_{k},\cdots, r_2, r_1+\sigma_{N+2} ) (2\leq k \leq N-1)$ and $\zeta(s_N, r_N, r_{N-1}, \cdots, r_2, r_1+\sigma_{N+2})$ are $(N+1)$-tuple zeta values of weight $\sum_{1 \leq i \leq N+2}\sigma_{i} + \sum_{1\leq j \leq N-1}\mu_{j}$, 

Let the permutation $\tau_{k}$ act on $M_{(k)}$ by permutating the index of $\left ( n_{j}\right)_{1\leq j \leq N+1}$. It is not difficult to see that
\[ \tau_{k} \cdot \widetilde{M}_{(k)} = M_{(k)}, (2 \leq k \leq N+1), \]
therefore 
\[ \sum_{\substack{n_j \in \mathbb{N}\smallsetminus\{0\}\\ 1 \leq j \leq N+1}} M_{(k)} = \sum_{\substack{n_j \in \mathbb{N}\smallsetminus\{0\}\\ 1 \leq  j \leq N+1}} \tau_{k}\cdot \widetilde{M}_{(k)} = \sum_{\substack{n_j\in \mathbb{N}\smallsetminus\{0\}\\ 1 \leq j \leq N+1}} \widetilde{M}_{(k)}. \]
We deduce that $\sum_{n_j\in \mathbb{N}\smallsetminus\{0\}} a^{k}_1$ (for each $2 \leq k \leq N$) and $\sum_{n_i\in \mathbb{N}\smallsetminus\{0\}} b^{N}_1$ are $\mathbb{Z}$-linear combinations of $(N+1)$-tuple zeta values.
\end{proof}

%For example, for a plane trivalent tree with one internal vertex,  the corresponding rooted tree illustrating the procedure of calculations is as follows. 

%\begin{forest}
%qtree edges
%[$P_{1}$
% [ $a^{1}_{1}$ 
% ]
%[ $b^{1}_{1}$ [ $a^{2}_{1}$  ]
                       %  [$b^{2}_{1}$ [$a^{3}_{1}$] [$b^{3}_{1}$] ]
% ] 
%]                          
%\end{forest}

Therefore $P_1$ is indeed a finite $\mathbb{Z}$-linear combination of $(N+1)$-tuple zeta values of weight $\sum_{1 \leq i \leq N+2}\sigma_{i} + \sum_{1\leq j \leq N-1}\mu_{j}$.

We can obtain similar results for other $P_{\gamma}$ and obtain $a^{i}_{\gamma} (1 \leq i \leq N+1) $. 
Therefore we get 
\[Z_{I,\nu}(\Gamma, \partial  \Gamma) = \sum_{n_1, \ldots, n_{N+1} \in \mathbb{N}\smallsetminus\{0\} } \sum_{1\leq i \leq N+1}\sum_{\gamma \in S_{d-1}}  a^{i}_{\gamma}\] 
\[ = \sum_{n_1, \ldots, n_{N+1} \in \mathbb{N}\smallsetminus\{0\} } \sum_{\alpha \in S_{d}}\sum_{\tilde{t}^{\alpha}_{i}} \frac{C_{\alpha,\tilde{t}^{\alpha}_{i}}}{n^{\tilde{t}^{\alpha}_1}_{\alpha \cdot 1}(n_{\gamma\cdot 1}+n_{\alpha\cdot 2})^{\tilde{t}^{\alpha}_2}\cdots (n_{\alpha \cdot 1}+\ldots + n_{\alpha \cdot d})^{\tilde{t}^{\alpha}_{d}}},\]
\[= \sum_{\alpha \in S_{d}}\sum_{\tilde{t}^{\alpha}_{i}}C_{\alpha,\tilde{t}^{\alpha}_{i}} \zeta (\tilde{t}^{\alpha}_1, \cdots, \tilde{t}^{\alpha}_d), \]
where $\alpha = \tau_i\cdot \gamma $ ($1 \leq i \leq N+1$ and $\gamma \in S_{d-1}$) is an element of $ S_d$.  
Note that $\tau_i$ ($1 \leq i \leq N+1$) and $\gamma \in S_{d-1}$ generate the symmetric group $S_d$ in the sense that
\[ S_{d} = S_{N+1} = \coprod^{N+1}_{i =1}\tau_{i}S_{N} = \coprod^{N+1}_{i =1}\tau_{i}S_{d-1}.\]

So we have finished the proof of the case $N = n+1$, $Z_{I,\nu}(\Gamma, \partial  \Gamma) $ is indeed a finite $\mathbb{Z}$-linear combination of $(N+1)$-tuple zeta values of weight $\sum_{1 \leq i \leq N+2}\sigma_{i} + \sum_{1\leq j \leq N-1}\mu_{j}$.
\\

   (\textbf{II}).\quad $\Gamma$ \textbf{is an arbitrary tree}.\\
The demonstration is quite similar to the previous proof for plane trivalent trees. Let $\Gamma$ be an arbitrary plane tree with $N$ internal vertices. For each internal vertex $w_j (1\leq j \leq N)$, the valency $val(w_j) = 3 + \beta_j (\beta_j \geq 0)$. $\Gamma$ has $N-1$ internal edges, $N+2 + \sum_{1\leq j \leq N}\beta_{j}$ external edges and $N+2 + \sum_{1\leq j \leq N}\beta_{j}$ external vertices $v_{i}$. Therefore the rank $d = \mathrm{rank}(H_{1}(\Gamma, \partial \Gamma))$ is not equal $N+1$ any more, but equal to
\[ d = N+1 + \sum_{1\leq j \leq N}\beta_{j}.\]

We give a clockwise order to the set of all external vertices $v_i$. We will also give an order for all internal vertices, such that  the internal vertex $w_{N}$, decorated by the variable $x_{N}$, is connected with the external vertices 
$$v_{N +1 +\sum_{1\leq j\leq N-1}\beta_{j} }, \ldots , v_{N+2 + \sum_{1\leq j\leq N}\beta_{j}}$$
 by external edges 
 $$\overrightarrow{e_{N+1+\sum_{1\leq j\leq N-1}\beta_{j} }} = (w_N\longrightarrow v_{N+1+\sum_{1\leq j\leq N-1}\beta_{j} }), \ldots ,\overrightarrow{e_{N+2 + \sum_{1\leq j\leq N}\beta_{j} }} = (w_N \longrightarrow v_{N+2+\sum_{1\leq j\leq N}\beta_{j} }).$$ 
We will again apply mathematical induction on the number of internal vertices.

\textbf{The initial step}: the case $N= 1$. Let $\beta + 3$ be the valency of the unique internal vertex of $\Gamma$.
If $\beta = 1$, the theorem is done due to Example \ref{ex:1'}. We need to prove the case if $\beta \ (\geq 1)$. Recall that in this case the rank $d$ of the tree is $ \beta + 2 $.

 The subdivision map is fixed, which means that for each (external) edge $e_{i} ( 1\leq i \leq 3+\beta)$, $\sigma_{i} -1 (\sigma_{i}\geq 1)$ points are added. The unique internal vertex is decorated by $x$, each external vertex $v_{i}$ is decorated by $x_{v_{i}}$. For each $e_i ( 1 \leq i \leq 2+\beta)$, the given sign $\nu_i$ equals $0$. Then the sign $\nu_{3+\beta}$ is forced to be $1$, since 
 \[\sum^{3+\beta}_{i=1} n_i = 0.\]

$$G_{I,\nu,\Gamma, \partial \Gamma}(\{x_{v}\}_{v\in \partial \Gamma}, 1) = \sum_{\substack{n_{1}+ \cdots + n_{3+\beta}=0,n_{i}\in \mathbb{Z}\smallsetminus \{0\};\\\mathrm{sgn}(n_{i})=(-1)^{\nu_i}, 1 \leq i \leq 2+\beta}}\frac{e^{2\pi i( n_{1}x_{v_{1}}+\cdots +n_{3+\beta}x_{v_{3+\beta}})}}{\prod^{3+\beta}_{i=1}|n_{i}|^{\sigma_{i}}}.
$$
Then 
$$ Z_{I,\nu}(\Gamma, \partial \Gamma) = G_{I,\nu,\Gamma, \partial \Gamma}(\{0\}_{v\in \partial \Gamma}, 1) =  \sum_{\substack{n_{1}+ \cdots + n_{3+\beta}=0,n_{i}\in \mathbb{Z}\smallsetminus \{0\};\\\mathrm{sgn}(n_{i})=(-1)^{\nu_i}, 1 \leq i \leq 2+\beta}}\frac{1}{\prod^{3+\beta}_{i=1}|n_{i}|^{\sigma_{i}}}$$
\[Z_{I,\nu}(\Gamma, \partial \Gamma) = \sum_{\substack{n_i \in \mathbb{N}\smallsetminus\{0\}\\1 \leq i \leq 2+\beta}}\frac{1}{\left(\prod^{2+\beta}_{i=1}n_{i}^{\sigma_{i}}\right)(n_{1}+\cdots+n_{2+\beta})^{\sigma_{3+\beta}}} .\]

\begin{note} If $0 \leq k \leq \beta$,
%\[Q^{k}\left(\begin{smallmatrix}
%  n_1 , & \cdots , & n_{k+1}; & n_{k+2}, & \cdots, & n_{2+\beta} \\    
%  _{k}t_{1} ,& \cdots,  & _{k}t_{k+1} ; & \sigma_{k+2}, &\cdots, &\sigma_{2+\beta}
%\end{smallmatrix}\right)\]
\[ Q^{(k)} \left( x_1, \cdots, x_{k+1}; x_{k+2}, \cdots, x_{\beta+2}; x_{\beta+3}\right)\]
\[= \frac{1}{\left(\prod^{k+1}_{j=1}n(1, j)^{x_j} \right) \left(\prod^{\beta+2}_{j=k+2}n_j^{x_j} \right) n(1, \beta+2)^{x_{\beta+3}}},\]
where each $x_j$ is a positive integer and the symbol $n(1,j)$ is defined in Notation (\ref{notation}).
\end{note}

%\[Q^{k}\left(t^{(k)}_{1} , \cdots,  t^{(k)}_{k+1} ; \sigma_{k+2}, \cdots, \sigma_{\beta+2}; \sigma_{\beta+3}\right)\]
%\[= \frac{1}{\left(\prod^{\beta+2}_{j=k+2}n_{j}^{\sigma_{j}}\right)n(1,\beta+2)^{\sigma_{\beta+3}}\left(\prod^{k+1}_{j=1} n(1, j)^{t^{(k)}_{j}}\right)},\]
%where $t^{(k)}_{j} (1\leq j \leq k+1)$ are variables and the upper-right mute symbol $(k)$ in $t^{(k)}_{j}$ signifies the dependence of $k$ and 

%$Q^{0}\left(\begin{smallmatrix}
 % n_1; & n_{2}, & \cdots, & n_{2+\beta} \\    
 %\sigma_{1} ; & \sigma_{2}, &\cdots, &\sigma_{2+\beta}
%\end{smallmatrix}\right) = Q$. 
In particular, let us consider
$$Q^{(0)}\left(\sigma_{1} ;  \sigma_{2}, \cdots, \sigma_{\beta+2}; \sigma_{\beta+3} \right)  = \frac{1}{\left(\prod^{2+\beta}_{i=1}n_{i}^{\sigma_{i}}\right)(n(1, \beta + 2)^{\sigma_{\beta+3}}},$$
then 
\[Z_{I,\nu}(\Gamma, \partial \Gamma) = \sum_{\substack{n_i \in \mathbb{N}\smallsetminus\{0\}\\1 \leq i \leq 2+\beta}} Q^{(0)}\left(\sigma_{1} ;  \sigma_{2}, \cdots, \sigma_{2+\beta};\sigma_{3+\beta} \right).\]

Now we show for $0 \leq q \leq \beta -1 $ how to express 
$$Q^{(q)}(x_1, \cdots, x_{q+1}; x_{q+2}, \cdots x_{\beta+2}; x_{\beta +3})$$
from $Q^{(q+1)}(\cdots; \cdots; \cdot)$ through several uses of Eisenstein's trick. We also show that $Q^{(\beta)}(\cdots)$ can be written
as a finite linear combination of classeical MZVs. Therefore we deduce the theorem by induction.

We define $q+2$ permutations as follows. 
\[ \tau_{1}^{(q)} = Id; \]
\[
  \tau_{j}^{(q)}= \bigl(\begin{smallmatrix}
    2+ q- (j-1) & 2+q-(j-2) &  \cdots & 1+q & 2+q\\    
    2+q & 2+q-(j-1) & \cdots &  q  & 1+q
  \end{smallmatrix}\bigr), \ 2 \leq j \leq q+2 .
\]

The upper right symbol $(q)$ of $ \tau_{j}^{(q)}$ signifies the dependence of the permutation on the number $q$.  

Let these permutations $\left(\tau^{(q)}_{j} \right)_{1\leq j \leq q+2}$ act on $Q^{(k)}(\cdots; \cdots; \cdot)$ by permutating the indices $i$ of $n_{i}$ without changing the exponents, 
\[\tau^{(q)}_j \cdot Q^{(k)} \left( x_1, \cdots, x_{k+1}; x_{k+2}, \cdots, x_{\beta+2}; x_{\beta+3}\right)\]
\[= \frac{1}{\left(\prod^{k+1}_{p =1}(\sum^{p}_{l=1} n_{\tau^{(q)}_{j}\cdot l})^{x_p} \right) \left(\prod^{\beta+2}_{p=k+2}n_{\tau^{(q)}_{j}\cdot p}^{x_p} \right) \left(\sum ^{\beta+2}_{p =1} n_{\tau^{(q)}_{j}\cdot p}\right)^{x_{\beta+3}}}.\]

For abbreviation, we denote for $1 \leq j \leq q$,
\[ Q^{(q+1)}_{\tau^{(q)}_j}\left(r_l ,s_l\right)_{1\leq l \leq j} = \tau^{(q)}_j \cdot Q^{(q+1)}\left(x_1, \cdots, x_{q+1-j}, s_j, r_j, \cdots, r_1; x_{q+3}, \cdots, x_{\beta +2}; x_{\beta+2} \right); \]

\[ Q^{(q+1)}_{\tau^{(q)}_{q+1}}\left(r_l ,s_l\right)_{1\leq l \leq q+1} = \tau^{(q)}_{q+1} \cdot Q^{(q+1)}\left( s_{q+1}, r_{q+1}, \cdots, r_1; x_{q+3}, \cdots, x_{\beta +2}; x_{\beta+2} \right); \]
\[ Q^{(q+1)}_{\tau^{(q)}_{q+2}}\left(r_l ,s_l\right)_{1\leq l \leq q+2} = \tau^{(q)}_{q+2} \cdot Q^{(q+1)}\left( s_{q+1}, r_{q+1}, \cdots, r_1; x_{q+3}, \cdots, x_{\beta +2}; x_{\beta+2} \right), \]
and for all $1 \leq j \leq q+2$, 
\[r_1 + s_1 = x_{q+2} + x_{q+1}; \quad r_l + s_l  =  s_{l-1} + x_{q-l+2}, \ 2 \leq l \leq j. \] 
Then for $1 \leq j \leq q+2$,
\[ \sum^{q+1-j}_{l=1}x_l + s_j + \sum^{j}_{l=1}r_l + \sum^{\beta+3}_{l= q+3}x_l = \sum ^{\beta+3}_{l=1}x_l.\]

We denote 
\[\sum Q^{(q+1)}_{\tau^{(q)}_1} = \sum_{r_1 + s_1 = x_{q+2} + x_{q+1}}C^{x_{q+2}-1}_{r_1 -1}Q^{(q+1)}_{\tau^{(q)}_1}\left(r_1 ,s_1\right), \]
%_{\substack{r_1, s_1}}
and if $2 \leq j \leq q+2$, 
\[\sum Q^{(q+1)}_{\tau^{(q)}_j} = \sum_{\substack{r_1 + s_1 = x_{q+2} + x_{q+1}; \\ r_l + s_l  =  s_{l-1} + x_{q-l+2}, \\ 2 \leq l \leq j}} C^{(s_{j-1})-1}_{r_j-1}C^{(x_{q-j+3})-1}_{(r_{j-1})-1}\cdots C^{x_{q+1}-1}_{r_1-1} Q^{(q+1)}_{\tau^{(q)}_j}\left(r_l ,s_l\right)_{1\leq l \leq j} \]
%_{\substack{r_l, s_l\\ 1 \leq l \leq j }}

We give a rooted tree illustrating the relation between the case $k=q$ and the case $k = q + 1$ by a repeated times of Eisenstein's trick. For the first time, $Eis\left(n^{x_{q+2}}_{q+2}, n(1, q+1)^{x_{q+1}}\right)$ is applied. Then for the level $2 \leq l \leq q+1$, $Eis\left(n^{s_{l-1}}_{q +2}, n(1, q-l+2)^{x_{q-l +2}} \right)$ is applied.
 
\begin{forest}
qtree edges
[$Q^{(q)}$
 [ $\sum Q^{(q+1)}_{Id}$ 
 ]
 [ $\cdot$ [ $\sum Q^{q+1}_{\tau_{2}^{(q)}}$  ]
                         [$\cdot$ [$\sum Q^{(q+1)}_{\tau_{3}^{(q)}}$] [$\cdot$
                            []
                            [$\cdots $
                              [$\sum Q^{q+1}_{\tau_{q+1}^{(q)}}$]
                              [$\sum Q^{q+1}_{\tau_{q+2}^{(q)}}$]
                            ]
                         ] ]
 ] 
]                          
\end{forest} 

In conclusion 
\[Q^{(q)} = \sum Q^{(q+1)}_{Id} + \sum Q^{(q+1)}_{\tau_{2}^{(q)}} + \cdots + \sum Q^{(q+1)}_{\tau_{q+1}^{(q)}} +\sum Q^{(q+1)}_{1,\tau_{q +2}^{(q)}}. \]

When $k = \beta $, 
%\[Q^{\beta}\left(\begin{smallmatrix}
%  n_1 , & \cdots , & n_{\beta+1}; &  n_{2+\beta} \\    
%  _{\beta}t_{1} ,& \cdots,  & _{\beta}t_{\beta+1} ; &  \sigma_{2+\beta}
%\end{smallmatrix}\right)\]
\[Q^{(\beta)}\left( x_{1} , \cdots,   x_{\beta+1} ;  x_{2+\beta}; x_{3+\beta}\right)
= \frac{1}{n_{\beta+2}^{x_{\beta+2}}n(1, 2+\beta)^{x_{3+\beta}}\left( \prod^{1+\beta}_{j=1} n(1, j)^{x_{j}}\right)}.\]

We define $\tau_1^{(\beta)} =Id$ and if $2 \leq j \leq \beta+2$
\[
  \tau_{j}^{(\beta)}= \bigl(\begin{smallmatrix}
    2 + \beta- (j-1) & 2+\beta-(j-2) &  \cdots & 1+\beta & 2+\beta\\    
    2 +\beta & 2+\beta -(j-1) & \cdots &  \beta   & 1+\beta
  \end{smallmatrix}\bigr) \in S_{2+\beta}.
\]

%\[Q^{\beta}\left( t^{(\beta)}_{1} , \cdots,   t^{(\beta)}_{\beta+1} ;  \sigma_{2+\beta}; \sigma_{3+\beta}\right)
%= \frac{1}{n_{\beta+2}^{\sigma_{\beta+2}}n(1, 2+\beta)^{\sigma_{3+\beta}}\left( \prod^{1+\beta}_{j=1} n(1, j)^{t^{(\beta)}_{j}}\right)}.\]
From the proof for plane trivalent trees (the calculation of $P_1$), we know that there exist a family of integers $\left( t^{\tau_{j}^{(\beta)}}_{i} \right)_{1\leq i, j \leq \beta+2} \in \mathbb{N} \smallsetminus \{0\}$ and $\left( C_{\tau_{j}^{(\beta)}, t^{\tau_{j}^{(\beta)}}_{i}} \right)_{i,j} \in \mathbb{Z}$ such that
%$$Q^{\beta}\left( t^{(\beta)}_{1} , \cdots,  t^{(\beta)}_{\beta+1} ;  \sigma_{2+\beta}; \sigma_{3+\beta}\right)$$
\[Q^{(\beta)}\left( x_{1} , \cdots,   x_{\beta+1} ;  x_{2+\beta}; x_{3+\beta}\right)\]
%$Q^{\beta}\left(\begin{smallmatrix}
%  n_1 , & \cdots , & n_{\beta+1}; &  n_{2+\beta} \\    
%  _{\beta}t_{1} ,& \cdots,  & _{\beta}t_{\beta+1} ; &  \sigma_{2+\beta}
%\end{smallmatrix}\right)$  

can be written as 
\[  = \sum_{\substack{\tau_{j}^{(\beta)}\\ 1 \leq j \leq \beta+2}}\sum_{\substack{t^{\tau_{j}(\beta)}_{i} \\ 1 \leq i \leq \beta +2}} \frac{C_{\tau_{j}^{(\beta)},(t^{\tau_{j}^{(\beta)}}_{i})_i}}{ n ^{t^{\tau_{j}^{(\beta)}}_1}_{\tau^{(\beta)}_{j} \cdot 1}(n_{\tau^{(\beta)}_{j}\cdot 1}+n_{\tau^{(\beta)}_{j}\cdot 2})^{t^{\tau_{j}^{(\beta)}}_2}\cdots (n_{\tau^{(\beta)}_{j} \cdot 1}+\ldots + n_{\tau^{(\beta)}_{j}\cdot (\beta+2)})^{t^{\tau_{j}^{(\beta)}}_{\beta+2}}},\]
where
\[ t^{\tau_j^{(\beta)}}_1 + \cdots t^{\tau_j^{(\beta)}}_{\beta+2} = \sum^{\beta+3}_{i=1}x_{i}, \quad \forall \  1 \leq j \leq \beta +2.\]
\\

Now we conclude by induction that
$$Q^{(0)}\left(
  \sigma_{1} ;  \sigma_{2}, \cdots, \sigma_{\beta+2}
\right) $$

can be written as 
\[  =  \sum_{\widehat{\tau}}\sum_{\tilde{t}^{\widehat{\tau}}_{i}} \frac{C_{\widehat{\tau},\tilde{t}^{\widehat{\tau}}_{i}}}{n^{\tilde{t}^{\widehat{\tau}}_1}_{\widehat{\tau} \cdot 1}(n_{\widehat{\tau}\cdot 1}+n_{\widehat{\tau}\cdot 2})^{\tilde{t}^{\widehat{\tau}}_2}\cdots (n_{\widehat{\tau} \cdot 1}+\ldots + n_{\widehat{\tau}\cdot (\beta+2})^{\tilde{t}^{\widehat{\tau}}_{\beta+2}}},\]
and
\[\tilde{t}^{\widehat{\tau}}_1 + \tilde{t}^{\widehat{\tau}}_2 + \cdots +\tilde{t}^{\widehat{\tau}}_{\beta+2} = \sum^{\beta+3}_{i=1} \sigma_{i},\]
where 
\[ \widehat{\tau} = \tau_{l_{\beta}}^{(\beta)}\tau_{l_{\beta-1}}^{(\beta-1)}\cdots \tau_{l_0}^{(0)}, \]
for each $0 \leq m \leq \beta$, we have $\tau_{1}^{(m)} = Id$ and, if $2 \leq l_m \leq m + 2 $, 
\[
  \tau_{l_m}^{(m)}= \bigl(\begin{smallmatrix}
    2+ m- (l_m -1) & 2+m-(l_m -2) &  \cdots & 1+m & 2+m\\    
    2+m & 2+m-(l_m -1) & \cdots &  m  & 1+m
  \end{smallmatrix}\bigr).
\]

\[ | \{\tau_{l_m}^{(m)} ; \quad 1 \leq l_{m} \leq m+2 , \quad 0 \leq m \leq \beta\} | = (\beta+2)(\beta+1)\cdots 2 = (\beta+2)! = |S_d|. \]

Therefore we have
\[ Z_{I, \nu}(\Gamma, \partial \Gamma) = \sum_{\widehat{\tau}}\sum_{\tilde{t}^{\widehat{\tau}}_{i}}C_{\widehat{\tau},\tilde{t}^{\widehat{\tau}}_{i}} 
\zeta\left( \tilde{t}^{\widehat{\tau}}_1, \cdots, \tilde{t}^{\widehat{\tau}}_{\beta +2} \right).\]
Hence theorem \ref{theo:1} holds for a tree with unique internal vertex of valency $\beta + 3$ ($\beta \geq 1$).

\begin{re} [See also Remark \ref{re:1} below]
We should mention that in the case of unique internal vertex, the generalized multiple zeta values are just Mordell-Tornheim zeta values. In \cite{BZ10} Bradley and Zhou proved that Mordell-Tornheim zeta values are $\mathbb{Q}$-linear combination of MZVs. Here we gave a new proof of Bradley-Zhou theorem. Moreover, our proof makes appearance the role of the symmetry group whose rank is equal to the rank of the graph. 
\end{re}

\textbf{The inductive step:}

What we will change for induction in the proof is the following :\\
Recall that $d = N+1+\sum_{1\leq j\leq N}\beta_{j}$, then let $d' = N +\sum_{1\leq j\leq N-1}\beta_{j}.$

In order to deduce the case $N = n+1$ to the case $N -1 = n$, we apply the following operation: we cut down the internal edge $f_{N-1}$, one of whose ends is the internal vertex $w_{N}$ and associate the new external vertex denoted as $v'_{N+1+\sum_{1\leq j\leq N-1}\beta_{j} } = v'_{d' + 1}$ and denote the new external edge as $e'_{N+1+\sum_{1\leq j\leq N-1}\beta_{j} } = e'_{d'+1}$ to which we associate $n_{e'_{N+1+\sum_{1\leq j\leq N-1}\beta_{j} }} = m_{N-1}$ and the subdivision $\mu_{N-1}$, then we build up a new plane tree $\Gamma '$ with $N-1 = n$ internal vertices and whose rank is $d' = d-1-\beta_{N}$, where $d$ is the rank of $\Gamma $. 

\[Z_{I,\nu}(\Gamma ', \partial \Gamma ') = \sum_{\substack{n_i \in \mathbb{N}\smallsetminus\{0\}, 1\leq i \leq d' }}\mathscr{O}_{I,\nu}(\Gamma ', \partial \Gamma '). \]
Then 
$$
Z_{I,\nu}(\Gamma , \partial \Gamma ) = G_{I,\nu,\Gamma,\partial \Gamma}(\{0\}_{v\in \partial \Gamma},1)= \sum_{\substack{n_i \in \mathbb{N}\smallsetminus\{0\},\\1\leq i \leq d}}\frac{1}{\prod^{d-N}_{k = d'-N }|n_{N+1+k}|^{\sigma_{N+1+k}}}\cdot \mathscr{O}_{I,\nu}(\Gamma ', \partial \Gamma ').
$$

Since the number of the internal vertices of $\Gamma '$ is n, then the theorem for $\Gamma '$ holds by the inductive hypothesis.

Hence we have the following equality: 
\[ \mathscr{O}_{I,\nu}(\Gamma ', \partial \Gamma ') = \sum_{\gamma \in S_{d'}} \sum_{t^{\gamma}_1 , \ldots ,t^{\gamma}_{d'} }\frac{\tilde{C}_{\gamma, t^{\gamma}_{i}}}{n^{t^{\gamma}_1}_{\gamma \cdot 1}(n_{\gamma\cdot 1}+n_{\gamma\cdot 2})^{t^{\gamma}_2}\cdots (n_{\gamma \cdot 1}+\ldots + n_{\gamma \cdot (d')})^{t^{\gamma}_{d'}}}, \]
where $$t^{\gamma}_1 + \ldots + t^{\gamma}_{d'} = \sum_{1\leq i \leq d'}\sigma_{i} + \sum_{1\leq j\leq N - 1}\mu_{j}, \quad \forall \gamma \in S_{d'}. $$
Now 
\[ P_{\gamma} = \frac{1}{\prod^{d-N}_{k = d'-N} |n_{N+1+k}|^{\sigma_{N+1+k}}}\frac{\tilde{C}_{\gamma,t^{\gamma}_i}}{n^{t^{\gamma}_1}_{\gamma \cdot 1}(n_{\gamma\cdot 1}+n_{\gamma\cdot 2})^{t^{\gamma}_2}\cdots (n_{\gamma \cdot 1}+\ldots + n_{\gamma \cdot (d')})^{t^{\gamma}_{d'}}}, \]
where $d' = N + \sum_{1\leq j \leq N-1}\beta_{j}$ and 
$$
\sum^{d+1}_{i =1}n_i = 0.
$$
Then 
\[ P_{\gamma} = \frac{1}{n_{d'+1}^{\sigma_{d'+1}}\cdots n_{d'+1+\beta_N}^{\sigma_{d'+1+\beta_N}}(n_1+\cdots + n_{d'+1+\beta_{N}})^{\sigma_{d'+2+\beta_N}}}\frac{\tilde{C}_{\gamma,t^{\gamma}_i}}{n^{t^{\gamma}_1}_{\gamma \cdot 1}(n_{\gamma\cdot 1}+n_{\gamma\cdot 2})^{t^{\gamma}_2}\cdots (n_{\gamma \cdot 1}+\ldots + n_{\gamma \cdot (d')})^{t^{\gamma}_{d'}}}, \]
and
\[Z_{I,\nu}(\Gamma , \partial \Gamma ) = \sum_{\substack{n_i \in \mathbb{N}\smallsetminus\{0\},\\1\leq i \leq d }}\sum_{\gamma \in S_{d'}} \sum_{t^{\gamma}_1 , \ldots ,t^{\gamma}_{d'}} P_{\gamma} . \]
Without losing generality, we can focus on $P_{1}$, where $1$ is the identity permutation.

\textbf{Calculation of} $P_1$. 
\[ P_{1} = \frac{1}{n_{d'+1}^{\sigma_{d'+1}}\cdots n_{d'+1+\beta_N}^{\sigma_{d'+1+\beta_N}}(n_1+\cdots + n_{d'+1+\beta_{N}})^{\sigma_{d'+2+\beta_N}}}\frac{\tilde{C}_{1,t^{1}_i}}{(n_{1}+\ldots + n_{ d'})^{t^{1}_{d'}}\cdots (n_{1}+n_{2})^{t^{1}_2}n^{t^{1}_1}_{1}}. \]

If $\beta_{N} =  0$, then we return to the calculation of $P_1$ for a plane trivalent tree. If $\beta_{N} > 0$, 
the number of monomials before the polynomial $(n_1+\cdots + n_{d'+1+\beta_{N}})^{\sigma_{d'+2+\beta_N}}$ is no more 1. 

We prove by the mathematical induction that $\mathscr{O}_{I,\nu}(\Gamma, \partial \Gamma)$ can be written in the form
\[ \sum_{\gamma \in S_d}\sum_{\substack{t^{\gamma}_{i} \\1 \leq i \leq d}} \frac{C_{\gamma}}{n^{t^{\gamma}_1}_{\gamma \cdot 1}(n_{\gamma\cdot 1}+n_{\gamma\cdot 2})^{t^{\gamma}_2}\cdots (n_{\gamma \cdot 1}+\ldots + n_{\gamma \cdot d})^{t^{\gamma}_{d}}} \]

with $d = d'+1+\beta_{N}$, and therefore deduce the theorem from  
$$
Z_{I,\nu}(\Gamma, \partial  \Gamma) = \sum_{\substack{n_{i}\in \mathbb{N}\smallsetminus\{0\},\\1 \leq i \leq d}} \mathscr{O}_{I,\nu}(\Gamma, \partial \Gamma).
$$

\begin{defi} If $0 \leq k \leq \beta_{N}$,
\[ P^{(k)}\left(x_{1}, \cdots, x_{d'+k} ; x_{d'+k+1}, \cdots, x_{d'+1+\beta_{N}}; x_{d'+2+\beta_{N}}\right)\]
\[= \frac{1}{ \left(\prod^{d'+1+\beta_N}_{j = d'+1+ k}n_{j}^{x_{j}}\right)\left( \prod^{d'+k}_{j=1} n(1, j)^{x_j}\right) n(1, d'+1+\beta_{N})^{x_{d'+2+\beta_N}}},\]
where $(x_j)_{1\leq j \leq d' + \beta_N}$ are all positive integers.
\end{defi}
In particular 
$$ P_1 = \tilde{C}_1 P^{(0)}\left( t^1_{1}, \cdots,  t^1_{d'}; \sigma_{d'+1}, \cdots, \sigma_{d'+1+\beta_{N}}; \sigma_{d'+2+\beta_{N}}\right).$$ 

%\textbf{Mathematical induction II}

If $k = \beta_{N}$,
\[P^{(\beta_{N})}\left( x_{1}, \cdots,  x_{d'+\beta_{N}} ; x_{d'+1+\beta_{N}}; x_{d'+2+\beta_{N}}\right)\]
\[= \frac{1}{ n_{\beta_N}^{x_{\beta_N}}\left( \prod^{d'+\beta_N}_{j=1} n(1, j)^{x_j}\right) n(1, d'+1+\beta_{N})^{x_{d'+2+\beta_N}}},\]

By the demonstration for any plane trivalent tree, $P^{(\beta_{N})}\left( x_{1}, \cdots,  x_{d'+\beta_{N}} ; x_{d'+1+\beta_{N}}; x_{d'+2+\beta_{N}}\right)$ can be written as 
\[  =  \sum_{\tau_{j}}\sum_{\tilde{t}^{\tau_{j}}_{i}} \frac{C_{\tau_{j},\tilde{t}^{\tau_{j}}_{i}}}{n^{\tilde{t}^{\tau_{j}}_1}_{\tau_{j} \cdot 1}(n_{\tau_{j}\cdot 1}+n_{\tau_{j}\cdot 2})^{\tilde{t}^{\tau_{j}}_2}\cdots (n_{\tau_{j} \cdot 1}+\ldots + n_{\tau_{j}\cdot (d'+1+\beta_{N}})^{\tilde{t}^{\tau_{j}}_{d'+1+\beta_{N}}}},\]
where the permutation $\tau_1 =Id$ and if $2 \leq j \leq d'+\beta_{N}+1$
\[
  \tau_{j}= \bigl(\begin{smallmatrix}
    d'+1 + \beta_{N}- (j-1) & d'+1+\beta_{N}-(j-2) &  \cdots & d'+\beta_{N} & d'+1+\beta_{N}\\    
    d'+1+\beta_{N} & d'+1+\beta_{N}-(j-1) & \cdots &  d'+\beta_{N}-1  & d'+\beta_{N}
  \end{smallmatrix}\bigr),
\]
and
\[\sum^{d'+\beta_N +1}_{l =1}\tilde{t}^{\tau_{j}}_l = \sum^{d'+\beta_N +2}_{l=1}x_l, \ \forall 1 \leq  j \leq d'+\beta_N +1,   \]
and the constant $C_{\tau_{j},\tilde{t}^{\tau_{j}}_{i}}$ is an integer.

Now we will explain for $0 \leq q \leq \beta_N -1 $ how to express 
$$P^{(q)}\left(x_{1}, \cdots, x_{d'+q} ; x_{d'+q+1}, \cdots, x_{d'+1+\beta_{N}}; x_{d'+2+\beta_{N}}\right)$$
from $P^{(q+1)}(\cdots; \cdots; \cdot)$ through several uses of Eisenstein's trick.

Let us define 
 $\tau_{1}^{(q)} = Id$ and, if $2 \leq j \leq d'+ 1+ q$, 
\[
  \tau_{j}^{(q)}= \bigl(\begin{smallmatrix}
    d'+1 + q- (j-1) & d'+1+q-(j-2) &  \cdots & d'+q & d'+1+q\\    
    d'+1+q & d'+1+q-(j-1) & \cdots &  d'+q-1  & d'+q
  \end{smallmatrix}\bigr).
\]

As in the previous step, we let these permutations act on $P^{(k)}(\cdots; \cdots; \cdot)$ by permuting the indices $i$ of $n_i$ without changing 
the exponents. 

For abbreviation, we denote for $1 \leq j \leq  d' + q -1 $,
\[ P^{(q+1)}_{\tau^{(q)}_j}\left(r_l ,s_l\right)_{1\leq l \leq j} = \tau^{(q)}_j \cdot P^{(q+1)}\left(x_1, \cdots, x_{q+d'-j}, s_j, r_j, \cdots, r_1; x_{d' + q +2}, \cdots, x_{d'+\beta_N +2}; x_{d'+\beta_N+2} \right); \]

\[ P^{(q+1)}_{\tau^{(q)}_{d' + q}}\left(r_l ,s_l\right)_{1\leq l \leq d'+ q} = \tau^{(q)}_{d' + q} \cdot P^{(q+1)}\left( s_{d'+q}, r_{d'+q}, \cdots, r_1; x_{d'+ q +2}, \cdots, x_{d'+\beta_N +1}; x_{d' + \beta_N+2} \right); \]
\[ P^{(q+1)}_{\tau^{(q)}_{d'+q+1}}\left(r_l ,s_l\right)_{1\leq l \leq d' + q} = \tau^{(q)}_{d'+q+1} \cdot P^{(q+1)}\left( s_{d'+q}, r_{d'+q}, \cdots, r_1; x_{d'+q+2}, \cdots, x_{d'+\beta_N +1}; x_{d'+\beta_N+2} \right). \]
And for all $1 \leq j \leq d'+ q$, 
\[r_1 + s_1 = x_{d'+q+1} + x_{d'+q}; \quad r_l + s_l  =  s_{l-1} + x_{d'+q-l+1}, \ 2 \leq l \leq j. \] 
Then for $1 \leq j \leq d'+ q$,
\[ \sum^{d'+q -j}_{l=1}x_l + s_j + \sum^{j}_{l=1}r_l + \sum^{d'+2+\beta_N}_{l= d'+q+2}x_l = \sum ^{d'+\beta_N +2}_{l=1}x_l.\]

We denote 
\[\sum P^{(q+1)}_{\tau^{(q)}_1} = \sum_{r_1 + s_1 = x_{d'+q+1} + x_{d'+q}}C^{x_{d'+q+1}-1}_{r_1 -1}P^{(q+1)}_{\tau^{(q)}_1}\left(r_1 ,s_1\right), \]
%_{\substack{r_1, s_1}}
and if $2 \leq j \leq d'+ q+1$, 
\[\sum P^{(q+1)}_{\tau^{(q)}_j} = \sum_{\substack{r_1 + s_1 = x_{d'+q+1} + x_{d'+q}; \\ r_l + s_l  =  s_{l-1} + x_{d'+q-l+1}, \\ 2 \leq l \leq j}} C^{(s_{j-1})-1}_{r_j-1}C^{(x_{d'+q-j+2})-1}_{(r_{j-1})-1}\cdots C^{x_{d'+q}-1}_{r_1-1} Q^{(q+1)}_{\tau^{(q)}_j}\left(r_l ,s_l\right)_{1\leq l \leq j} \]
%_{\substack{r_l, s_l\\ 1 \leq l \leq j }}

We give a rooted tree illustrating the relation between the case $k=q$ and the case $k = q + 1$ by a repeated use of Eisenstein's trick. For the first time, $Eis\left(n^{x_{d'+q+1}}_{d'+q+1}, n(1, d'+q)^{x_{d'+q}}\right)$ is applied. Then for the level $2 \leq l \leq q+1$, $Eis\left(n^{s_{l-1}}_{d'+q +1}, n(1, d'+q-l+1)^{x_{d'+q-l +1}} \right)$ is applied.
 
\begin{forest}
qtree edges
[$P^{(q)}$
 [ $\sum P^{(q+1)}_{Id}$ 
 ]
 [ $\cdot$ [ $\sum P^{(q+1)}_{\tau_{2}(q)}$  ]
                         [$\cdot$ [$\sum P^{(q+1)}_{\tau_{3}(q)}$] [$\cdot$
                            []
                            [$\cdots $
                              [$\sum P^{(q+1)}_{\tau_{d'+q}^{(q)}}$]
                              [$\sum P^{(q+1)}_{\tau_{d'+q +1}^{(q)}}$]
                            ]
                         ] ]
 ] 
]                          
\end{forest}

Then in conclusion, we can write
\[P^{(q)} = \sum P^{(q+1)}_{Id} + \sum P^{(q+1)}_{\tau_{2}^{(q)}} + \cdots + \sum P^{(q+1)}_{\tau_{d'+q}^{(q)}} +\sum P^{(q+1)}_{\tau_{d'+q +1}^{(q)}}. \]

By induction, we conclude that
$$P^{(0)}\left( t^1_{1}, \cdots,  t^1_{d'}; \sigma_{d'+1}, \cdots, \sigma_{d'+1+\beta_{N}}; \sigma_{d'+2+\beta_{N}}\right) $$
 can be written as 
\[  =  \sum_{\widehat{\tau}}\sum_{\tilde{t}^{\widehat{\tau}}_{i}} \frac{C_{\widehat{\tau},\tilde{t}^{\widehat{\tau}}_{i}}}{n^{\tilde{t}^{\widehat{\tau}}_1}_{\widehat{\tau}_{k} \cdot 1}(n_{\widehat{\tau}\cdot 1}+n_{\widehat{\tau}\cdot 2})^{\tilde{t}^{\widehat{\tau}}_2}\cdots (n_{\widehat{\tau} \cdot 1}+\ldots + n_{\widehat{\tau}\cdot (d'+1+\beta_{N}})^{\tilde{t}^{\widehat{\tau}}_{d'+1+\beta_{N}}}}\]
\[ = \sum_{\widehat{\tau}}\sum_{\tilde{t}^{\widehat{\tau}}_{i}} C_{\widehat{\tau},\tilde{t}^{\widehat{\tau}}_{i}}\zeta (\tilde{t}^{\widehat{\tau}}_1, \tilde{t}^{\widehat{\tau}}_2, \cdots , \tilde{t}^{\widehat{\tau}}_{d'+1+\beta_{N}})  ,\] 
where 
\[ \widehat{\tau} = \tau_{l_{\beta_N}}^{(\beta_N)}\tau_{l_{\beta_{N-1}}}^{(\beta_{N}-1)} \cdots \tau_{l_0}^{(0)}, \]
where $\tau_{1}^{(m)} = Id$ and, if $2 \leq l_m \leq d'+ 1+ m$, $0 \leq m \leq \beta_{N}$,
\[
  \tau_{l_m}^{(m)}= \bigl(\begin{smallmatrix}
    d'+1 + m- (l_m-1) & d'+1+m-(l_m-2) &  \cdots & d'+m & d'+1+m\\    
    d'+1+m & d'+1+m-(l_m-1) & \cdots &  d'+m-1  & d'+m
  \end{smallmatrix}\bigr).
\]
Let the set 
\[ D = \{ \widehat{\tau} = \prod_{0 \leq m \leq \beta_{N}}\tau_{l_m}(m) ; \quad 2 \leq l_m \leq d'+ 1+ m \}, \]
\[ |D| = (d'+1+\beta_{N})(d'+\beta_{N})\cdots (d'+1). \]
For other $\gamma \in S_{d'}$, we can obtain the same result of $P_{\gamma}$, and 
\[ S_{d} = \coprod_{\widehat{\tau}\in D}\widehat{\tau} \cdot S_{d'} .\]
\end{proof} 

\begin{re}
Eisenstein's trick has been formalized by Sczech in his theory of Eisenstein cocycles. It also
appears in one of the proofs of the shuffle relations for MZV's. This is probably no coincidence.
\end{re}

\paragraph{\textbf{Proof of Theorem \ref{thm:poly}}}
\begin{proof} 
For a general tree $\Gamma$, by the proof of Theorem \ref{theo:1}, $\mathscr{O}_{I,\nu}(\Gamma, \partial \Gamma)$ can be written as
\[ \sum_{\gamma \in S_d}\sum_{\substack{t^{\gamma}_{j} \\1 \leq j \leq d}} \frac{C_{\gamma,t^{\gamma}_j}}{n^{t^{\gamma}_1}_{\gamma \cdot 1}(n_{\gamma\cdot 1}+n_{\gamma\cdot 2})^{t^{\gamma}_2}\cdots (n_{\gamma \cdot 1}+\ldots + n_{\gamma \cdot d})^{t^{\gamma}_{d}}} \]

with $d$ the rank of $\Gamma$, and therefore conclude that 
$$
Z_{I,\nu}(\Gamma, \partial  \Gamma) = \sum_{\substack{n_{j}\in \mathbb{N}\smallsetminus\{0\},\\1 \leq j \leq d}} \mathscr{O}_{I,\nu}(\Gamma, \partial \Gamma).
$$
By careful observation, we get an expression
\[G_{I,\nu,\Gamma, \partial \Gamma}(\{x_{v}\}_{v\in \partial \Gamma}, 1) = \sum_{\substack{n_{j}\in \mathbb{N}\smallsetminus\{0\},\\1 \leq i \leq d}}e^{2\pi i(\sum^{d+1}_{j=1}n_jx_{v_{j}})}\left(\sum_{\gamma \in S_d}\sum_{\substack{t^{\gamma}_{j} \\1 \leq i \leq d}} \frac{C_{\gamma,t^{\gamma}_j}}{n^{t^{\gamma}_1}_{\gamma \cdot 1}(n_{\gamma\cdot 1}+n_{\gamma\cdot 2})^{t^{\gamma}_2}\cdots (n_{\gamma \cdot 1}+\ldots + n_{\gamma \cdot d})^{t^{\gamma}_{d}}}\right), \]
where $C_{\gamma,t^{\gamma}_{j}} \in \mathbb{Z}$.
In fact, the Eisenstein trick is applied for the denominators, during these repeated operations the numerator remains unchanged.
A simple calculation gives
\[ \sum^{d+1}_{j=1} n_jx_{v_{j}} = \sum^{d}_{l=1}(\sum^{l}_{j=1}n_{j})(x_{v_{l}}-x_{v_{l+1}}).\]
Let 
$$z_{j} = e^{2\pi i (x_{v_{j}}-x_{v_{j+1}})} , \quad 1 \leq j \leq d, $$
and 
$$ z_{\gamma \cdot j} = e^{2\pi i (x_{v_{\gamma \cdot j}}-x_{v_{\gamma \cdot (j+1)}})}, \quad \gamma \in S_{d},$$
then 
\[G_{I,\nu,\Gamma, \partial \Gamma}(\{x_{v}\}_{v\in \partial \Gamma}, 1) = \sum_{\substack{n_{j}\in \mathbb{N}\smallsetminus\{0\},\\1 \leq j \leq d}}\sum_{\gamma \in S_d}\sum_{\substack{t^{\gamma}_{i} \\1 \leq j \leq d}} \frac{C_{\gamma,t^{\gamma}_j}\prod^{d}_{i=1}z_{\gamma\cdot j}^{\sum^{j}_{l=1}n_{\gamma \cdot l}}}{n^{t^{\gamma}_1}_{\gamma \cdot 1}(n_{\gamma\cdot 1}+n_{\gamma\cdot 2})^{t^{\gamma}_2}\cdots (n_{\gamma \cdot 1}+\ldots + n_{\gamma \cdot d})^{t^{\gamma}_{d}}},\]

\[= \sum_{\gamma \in S_d}\left(\sum_{\substack{t^{\gamma}_{j} \\1 \leq j \leq d}}\sum_{\substack{n_{j}\in \mathbb{N}\smallsetminus\{0\},\\1 \leq j \leq d}}C_{\gamma,t^{\gamma}_j}\frac{\prod^{d}_{j=1}z_{\gamma\cdot j}^{\sum^{j}_{l=1}n_{\gamma \cdot l}}}{n^{t^{\gamma}_1}_{\gamma \cdot 1}(n_{\gamma\cdot 1}+n_{\gamma\cdot 2})^{t^{\gamma}_2}\cdots (n_{\gamma \cdot 1}+\ldots + n_{\gamma \cdot d})^{t^{\gamma}_{d}}}\right), \]
finally we obtain
\[G_{I,\nu,\Gamma, \partial \Gamma}(\{x_{v}\}_{v\in \partial \Gamma}, 1) = \sum_{\gamma \in S_d}\left(\sum_{\substack{t^{\gamma}_{j} \\1 \leq j \leq d}}C_{\gamma,t^{\gamma}_j} Li_{t^{\gamma}_{1},\cdots, t^{\gamma}_{d}}(z_{\gamma\cdot 1},\cdots , z_{\gamma \cdot d})\right).\]

In conclusion, $G_{I,\nu,\Gamma, \partial \Gamma}(\{x_{v}\}_{v\in \partial \Gamma}, 1) $ is indeed a finite $\mathbb{Z}$-linear combination of multiple polylogarithms evaluated at some $N$-th roots of unity.
\end{proof}

\begin{re} \label{re:1}
We should mention that Example (\ref{ex:1}) and Example (\ref{ex:1'}) are typical examples of Mordell-Tornheim zeta values (\cite{BZ10}), which are defined as 
\[T(s_1,\ldots, s_r; s) := \sum^{\infty}_{m_1 =1}\ldots \sum^{\infty}_{m_r =1}\frac{1}{m^{s_1}_1\ldots m^{s_r}_r(m_1 + \ldots +m_r)^{s}},\]
where $s_1, \ldots, s_r$ and $s$ are complex numbers with $s_1 + \ldots +s_r + s = w$ and $r$ is the depth and $w$ is the weight.

For given subdivision map and sign respectively for $\Gamma_1$ and $\Gamma'_1$, it is not difficult to see that
\[Z_{\mathbb{Z},\nu}(\Gamma_1,\partial \Gamma_1) = T(\sigma_1,\sigma_2;\sigma_3); \quad 
Z_{\mathbb{Z},\nu}(\Gamma'_1,\partial \Gamma'_1) = T(\sigma_1,\sigma_2,\sigma_3;\sigma_4).\] 

We should point out that Bradley and Zhou (\cite{BZ10}) have proven that any Mordell-Tornheim sum with positive integer arguments can be expressed as a rational linear combination of multiple zeta values of the same weight and depth. In fact given a plane tree $\mathrm{T}$ with only one internal vertex and $m$ edges, given a subdivision map $\underline{k} = (\sigma_i)_{1\leq i \leq m}$ and an appropriate sign $\nu$, we will always have 
\[Z_{\mathbb{Z},\nu}(\mathrm{T},\partial \mathrm{T} ) = T(\sigma_1,\sigma_2,\ldots, \sigma_{m-1};\sigma_m). \]
In this sense, Mordell-Tornheim zeta values are special cases of our generalized multiple zeta values associated to special graphs. 
Besides, we give details of proof for Example (\ref{ex:1}) and Example (\ref{ex:1'}) and keep a uniform way of demonstration as also shown in Example (\ref{ex:2}), in order to make readers to pay attention to the appearance of a permutation group acting on the indices.

However when the given graph has more internal vertices, our generalized multiple zeta value is no longer a Mordell-Tornheim zeta value, nor
a generalized Witten zeta value or their generalization- generalized zeta value associated to root systems, defined and studied by
Komori, Matsumoto, and Tsumura. For example, the generalized zeta value associated to $sl(l+1)$ (\cite{MT06}) is 
\[\zeta_{sl(l+1)}(\underline{s}) = \sum^{\infty}_{m_1,\ldots,m_r =1}\prod^{l}_{j=1}\prod^{l-j+1}_{k=1}\left(\sum^{j+k-1}_{t=k}m_{t} \right)^{s_{jk}}.\]

\[\zeta_{sl(4)}(s_1,s_2,s_3,s_4,s_5,s_6) =  \sum^{\infty}_{m_1,m_2,m_3=1}\frac{1}{m^{s_1}_1m^{s_2}_2m^{s_3}_3(m_1 +m_2)^{s_4}(m_2+m_3)^{s_5}(m_1+m_2+m_3)^{s_6}}.\]
They define (section 5 of (\cite{MT06}))
\[\mathcal{T}(s_1,s_2,s_3,s_4,s_5) = \zeta_{sl(4)}(s_1,s_2,s_3,s_4,0,s_5). \]
Then give some evaluation of $\mathcal{T}(s_1,s_2,s_3,s_4,s_5)$ for special $s_1,\ldots, s_5$, which are indeed a linear combination of $3$-depth MZVs. Example (\ref{ex:2}) shows that
\[Z_{I,\nu}(\Gamma_2,\partial \Gamma_2) = G_{I,\nu,\Gamma_{2}, \partial \Gamma_{2}}(\{0\}_{v\in S}, 1) =\sum_{n_{1},n_{2},n_{3}\in \mathbb{N}\smallsetminus\{0\}}\frac{1}{n_{3}^{\sigma_{3}}(n_{1}+n_{2}+n_{3})^{\sigma_{4}}}\frac{1}{n_{1}^{\sigma_{1}}n_{2}^{\sigma_{2}}(n_{1}+n_{2})^{\mu_{1}}}\]
is a finite $\mathbb{Z}$-linear combination of triple zeta values for general $\sigma_i$ and $\mu_1$. And 
\[Z_{I,\nu}(\Gamma_2,\partial \Gamma_2) = \mathcal{T}(\sigma_1,\sigma_2,\sigma_3,\mu_1,\sigma_4).\]
Moreover when the given graph is no longer a plane trivalent tree, our generalized multiple zeta value is not contained in work of Komori, Matsumoto, and Tsumura. Besides our higher plectic Green functions can be related to multiple polylogarithms.  Therefore our construction is quite new.
\end{re}

\vspace{1cm}

\newpage

\vspace{2cm}

\begin{thebibliography}{99}

%% Ngo's style
\providecommand{\natexlab}[1]{#1}
\providecommand{\url}[1]{\texttt{#1}}
\expandafter\ifx\csname urlstyle\endcsname\relax
  \providecommand{\doi}[1]{doi: #1}\else
  \providecommand{\doi}{doi: \begingroup \urlstyle{rm}\Url}\fi
  
 \bibitem[Ai18]{Ai18}
 Xiaohua Ai.
 \newblock{Arithmetic of values of $L$-functions and generalized multiple zeta values over number fields}.
 \newblock{Ph.D thesis of Sorbonne Univerisité, 2018}.

\bibitem[BKL18]{BKL18}
Alexander Beilinson, Guido Kings and Andrey Levin.
\newblock{Topological polylogarithms and p-adic interpolation of $L$-values of totally real fields}.
\newblock Math. Ann. 371 (2018), no. 3-4, 1449-1495.

\bibitem[BZ10]{BZ10}
D.M. Bradley, X. Zhou.
\newblock{On Mordell-Torenheim sums and multiple zeta values}.
\newblock Ann. Sci. Math. Qu\'ebec 34 (2010), no.1, 15-23.

\bibitem[Brow12]{Brow12}
Francis Brown.
\newblock {Mixted {T}ate {M}otives over {S}pec {Z}}.
\newblock Annals of Math. (2) Vol. 175. 2012, no. 2, 946-976.

\bibitem[Brow13]{Brow13}
Francis Brown.
\newblock {Dedekind {Z}eta motives for totally real fields}.
\newblock Inventiones Mathematicae, Vol. 194. Springer, 2013, 257-311.

\bibitem[Brow14]{Brow14}
Francis Brown.
\newblock {Multiple Modular Values and the relative completion of the fundamental group of $M_{1,1}$}.
\newblock {arXiv}:1407.5167, 2014.

\bibitem[GKZ06]{GKZ06}
H. Gangl, M. Kaneko and D. Zagier.
\newblock {Double Zeta Values and modular forms}.
\newblock Automorphic forms and zeta functions, 71–106, World Sci. Publ., Hackensack, NJ, 2006.

\bibitem[Gon94]{Gon94}
Alexander Goncharov.
\newblock {Polylogarithms and Motivic Galois groups}.
\newblock Motives(Seattle, 1991)Proc.{S}ympos. {P}ure {M}ath. 55, part 2, Amer. {M}ath. {S}oc. Providence, 1994.

\bibitem[Gon01]{Gon01}
Alexander Goncharov.
\newblock {Multiple polylogarithms and mixed {T}ate motives}.
\newblock math. AG/0103059, 2001.

\bibitem[Gon00]{Gon00}
Alexander Goncharov.
\newblock {Multiple $\zeta$ values, Galois groups and geometry of modular varieties}.
\newblock {E}uropean {C}ongress of {M}athematics, Progress in Math, {V}ol I (Barcelona, 2000), 361–392, Progr. Math., 201, Birkhäuser, Basel, 2001.


\bibitem[Gon05]{Gon05}
Alexander Goncharov.
\newblock {Galois symmetries of fundamental groupoids and noncommutative geometry}.
\newblock Duke Math. J. 128 (2005), no. 2, 209–284.

\bibitem[Gon08]{Gon08}
Alexander Goncharov.
\newblock {Euler complexes and geometry of modular varieties}.
\newblock Geom. Funct. Anal. 17 (2008), no. 6, 1872–1914. 

\bibitem[Gon16]{Gon16}
Alexander Goncharov.
\newblock {Hodge correlators}.
\newblock Journal für die reine und angewandte {M}athematik, Published Online: 2016-07-28.

\bibitem[Hain94]{Hain94}
Richard Hain.
\newblock {Classical Polylogarithms}.
\newblock Motives(Seattle, 1991)Proc.{S}ympos. {P}ure {M}ath. 55, part 2, Amer. {M}ath. {S}oc. Providence, 1994.

\bibitem[Hec17]{Hec17}
E. Hecke.
\newblock {Uber Die Kroneckersche Grenzformel Für Reelle Quadratische Körper Und Die Klassenzah Relativ-Abelscher Körper}.
\newblock Verhandl. Naturforsch. Ges. Basel, pp. 363-372, Vol. 28, 1917.

\bibitem[MT06]{MT06}
K. Matsumoto and H. Tsumura.
\newblock{On Witten multiple zeta functions associated with seminsimple Lie algebras I}.
\newblock Ann. Inst. Fourier Grenoble 56 (2006), 1457–1504.

\bibitem[Nor95]{Nor95}
Madhav Nori
\newblock {Some Eisenstein Cohomology Classes for the Integral Unimodular Group}.
\newblock Proceedings of the International Congress of Mathematical Sciences, vols, 1,2, (Zürich, 1994). Birkhäuser, Basel, 1995.

\bibitem[NS16]{NS16}
Jan Nekov\'a\v r and Antony Scholl.
\newblock {Introduction to {P}lectic {C}ohomology}.
\newblock Contempary {M}athematics, Advances in the theory of automorphic forms and their $L$-functions, 321–337, Contemp. Math., 664, Amer. Math. Soc., Providence, RI, 2016.

\bibitem[NS]{NS}
Jan Nekov\'a\v r and Antony Scholl.
\newblock {private communication}.
\newblock 2016.

%\bibitem[PP01]{PP01}
%P. Pasles and W. De Azevedo Pribitkin.
%\newblock {A generalization of the {L}ipschitz summation formula and some applications}.
%\newblock Proceedings of the AMS, Vol. 129, 2001.

\bibitem[Scz93]{Scz93}
Robert Sczech.
\newblock {Eisenstein group cocycles for $GL_n$ and values of {L}-functions}.
\newblock Inventiones mathematicae, Vol. 113, 1993, no. 3, 581-616.

\bibitem[Zag86]{Zag86}
Don Zagier.
\newblock {Hyperbolic manifolds and special values of {D}edekind zeta functions}.
\newblock Inventiones mathematicae, 83 (1986), no. 2, 285–301. 

\bibitem[Zag94]{Zag94}
Don Zagier.
\newblock {Values of zeta functions and their applications}.
\newblock First {E}uropean {C}ongress of {M}athematics, In: Joseph A., Mignot F., Murat F., Prum B., Rentschler R. (eds) First European Congress of Mathematics Paris, July 6–10, 1992. Progress in Mathematics, vol 120. Birkhäuser Basel.



%\bibitem[Drinfeld(2002)]{Drinfeld:2002vda}
%Vladimir Drinfeld.
%\newblock {On the Grinberg - Kazhdan formal arc theorem}.
%\newblock {arXiv}:0203263v1, 2002.

%% Milne's style
%\bibitem[\protect\astroncite{Fuchs}{1970}]{fuchs1970}
%{\sc Fuchs, L.} 1970.
%\newblock Infinite abelian groups. {V}ol. {I}.
%\newblock Pure and Applied Mathematics, Vol. 36. Academic Press, New York.

%\bibitem[\protect\astroncite{Brown}{2012}]{Brow12}
%{\sc Brown, F.} 2012.
%\newblock Mixted {T}ate {M}otives over {S}pec {Z}. 
%\newblock Annals of Math, Vol. 175. Princeton.

%\bibitem[\protect\astroncite{Brown}{2013}]{Brow13}
%{\sc Brown, F.} 2013.
%\newblock Dedekind {Z}eta motives for totally real fields. 
%\newblock Inventiones Mathematicae, Vol. 194. Springer. 

%\bibitem[\protect\astroncite{NS}{2016}]{NS16}
%{\sc Nekovar, J. and Scholl, A.} 2016.
%\newblock Introduction to {P}lectic {C}ohomology. 
%\newblock Contempary {M}athematics.

%%% display without the package natbib
%\bibitem{latexcompanion} 
%Michel Goossens, Frank Mittelbach, and Alexander Samarin. 
%\textit{The \LaTeX\ Companion}. 
%Addison-Wesley, Reading, Massachusetts, 1993.
 

\end{thebibliography}
\end{document}